\documentclass[11pt]{article}

\usepackage{latexsym}
\usepackage{amssymb}
\usepackage{amsthm}
\usepackage{amscd}
\usepackage{amsmath}
\usepackage{mathrsfs}
\usepackage{graphicx}
\usepackage{hyperref}
\usepackage{shuffle}
\usepackage{mathtools}
\usepackage[heads=vee]{diagrams} 
\usepackage[bbgreekl]{mathbbol}
\usepackage{ytableau}
\usepackage{bm}

\usepackage[all]{xy}
\input xy \xyoption{frame}
\xyoption{dvips}

\usepackage[colorinlistoftodos]{todonotes}
\usetikzlibrary{chains,scopes}

\DeclareSymbolFontAlphabet{\mathbb}{AMSb}
\DeclareSymbolFontAlphabet{\mathbbl}{bbold}

\theoremstyle{definition}
\newtheorem* {theorem*}{Theorem}
\newtheorem* {conjecture*}{Conjecture}
\newtheorem{theorem}{Theorem}[section]

\newtheorem{problem}[theorem]{Problem}
\theoremstyle{definition}

\newtheorem* {example*}{Example}

\newtheorem{lemma}[theorem]{Lemma}
\theoremstyle{definition}
\newtheorem{definition}[theorem]{Definition}
\theoremstyle{definition}

\newtheorem{proposition}[theorem]{Proposition}
\newtheorem{corollary}[theorem]{Corollary}

\newtheorem *{remark}{Remark}
\theoremstyle{definition}
\newtheorem {example}[theorem]{Example}
\theoremstyle{definition}

\theoremstyle{definition}

\theoremstyle{definition}

\xyoption{dvips}

\def\({\left(}
\def\){\right)}

\newcommand{\QQ}{\mathbb{Q}}

\newcommand{\cO}{\mathcal{O}}
\newcommand{\cR}{\mathcal{R}}

\newcommand{\cS}{\mathcal{S}}

\def\cS{\mathcal{S}}

\def\NN{\mathbb{N}}

\def\End{\mathrm{End}}

\def\RR{\mathbb{R}}
\def\ZZ{\mathbb{Z}}

\def\ch{\mathrm{ch}}
\def\spanning{\textnormal{-span}}

\def\fk{\mathfrak}

\def\barr{\begin{array}}
\def\earr{\end{array}}
\def\ba{\begin{aligned}}
\def\ea{\end{aligned}}
\def\be{\begin{equation}}
\def\ee{\end{equation}}

\def\qquand{\qquad\text{and}\qquad}
\def\quand{\quad\text{and}\quad}

\def\hs{\hspace{0.5mm}}

\def\ds{\displaystyle}

\def\id{\mathrm{id}}
\def\PP{\mathbb{P}}

\def\ben{\begin{enumerate}}
\def\een{\end{enumerate}}

\def\hs{\hspace{0.5mm}}

\def\Des{\mathrm{Des}}

\def\c{\textbf{c}}

\renewcommand{\r}[1]{\textcolor{red}{#1}}

\newcommand{\cA}{\mathcal{A}}
\newcommand{\cB}{\mathcal{B}}

\def\arcstart{\ \xy<0cm,-.15cm>\xymatrix@R=.1cm@C=.3cm }
\newcommand{\arcstartc}[1]{\ \xy<0cm,-.15cm>\xymatrix@R=.1cm@C=#1cm}

\def\t{\mathbf{t}}
\def\r{\mathbf{r}}

\def\sC{\mathscr{C}}

\def\kk{\mathbbl{k}}

\def\bfW{\textbf{W}}
\def\bfI{\textbf{I}}

\def\FC{\textsf{FC}}

\def\Words{\mathbb{W}}

\def\antipode{{\tt S}}

\def\dfslash{\hs/\hspace{-1.0mm}/\hs}
\def\dbslash{\hs\backslash\hspace{-1mm}\backslash\hs}

\newcommand{\Sym}{\textsf{Sym}}
\newcommand{\WQSym}{\textsf{WQSym}}
\newcommand{\FQSym}{\textsf{FQSym}}
\newcommand{\Shuffle}{\textsf{Shuffle}}

\def\QSym{\textsf{QSym}}

\def\Sym{\textsf{Sym}}

\def\zetaq{\zeta_{\QSym}}

\def\r{{\tt r}}
\def\c{{\tt c}}
\def\t{{\tt t}}

\def\P{\textsf{P}}
\def\Packed{\mathbb{W}_\P}
\def\bfP{{\textbf{W}}_\P}

\def\XX{\mathbb{X}}

\def\Peak{\mathrm{Peak}}
\def\Valley{\mathrm{Val}}
\def\OQSym{\cO\QSym}
\def\OSym{\cO\Sym}

\def\ch{\operatorname{char}}
\def\sort{\mathrm{sort}}

\def\rpk{\flat}

\newcommand{\word}{\operatorname{word}}

\def\bfPi{\mathbf{\Pi}}
\def\st{\operatorname{fl}}
\def\fkR{{r}}

\usepackage{fullpage}

\numberwithin{equation}{section}
\allowdisplaybreaks[1]
\UseCrayolaColors

\makeatletter
\renewcommand{\@makefnmark}{\mbox{\textsuperscript{}}}
\makeatother

\begin{document}
\title{Bialgebras for Stanley symmetric functions}
\author{
Eric Marberg
\\ Department of Mathematics \\  Hong Kong University of Science and Technology \\ {\tt eric.marberg@gmail.com}
}

\date{}

\maketitle

\begin{abstract}
We construct a non-commutative, non-cocommutative, graded bialgebra $\mathbf{\Pi}$
with a basis indexed by the permutations in all finite symmetric groups.
Unlike the formally similar Malvenuto-Poirier-Reutenauer Hopf algebra,
this bialgebra does not have finite graded dimension.
After giving formulas for the product and coproduct,
we show that there is a natural morphism from $\mathbf{\Pi}$ to the algebra of quasi-symmetric functions,
under which the image of a permutation is its associated Stanley symmetric function.
As an application, we use this morphism to derive some new enumerative identities.
We also describe analogues of $\mathbf{\Pi}$ for the other classical types.
In these cases, the relevant objects are module coalgebras rather than bialgebras,
but there are again natural morphisms to the quasi-symmetric functions,
under which the image of a signed permutation
is the corresponding Stanley symmetric function of type B, C, or D.
\end{abstract}

\setcounter{tocdepth}{2}
\tableofcontents

\section{Introduction}

Fix a positive integer $n$ and let $S_n$ denote the symmetric group of permutations of $\{1,2,\dots,n\}$,
which we write in one-line notation as words $\pi=\pi_1\pi_2\cdots \pi_n$
containing each $i \in \{1,2,\dots,n\}$ as a letter exactly once.
The \emph{right weak order} for $S_n$ is the partial order whose covering relations
are \[\pi_1 \cdots \pi_i \pi_{i+1}\cdots \pi_n \lessdot \pi_1 \cdots \pi_{i+1} \pi_{i}\cdots \pi_n \quad\text{whenever}\quad\pi_i < \pi_{i+1}.\]
A \emph{reduced word} for $\pi \in S_n$ corresponds to a maximal chain from the identity permutation $123\cdots n$ 
to $\pi$ in this order,
and the \emph{length} $\ell(\pi)$ of a permutation $\pi$ is the number of covering relations in any such chain.
We denote the number of reduced words for $\pi \in S_n$ by $\fkR(\pi)$.

The original motivation for this paper comes from a sequence of identities
relating counts of reduced words for certain permutations.
Given a word $w=w_1w_2\cdots w_n$ with distinct integer letters,
define the associated \emph{flattened word} by
$\st(w) = \phi(w_1)\phi(w_2)\cdots \phi(w_n) \in S_n$ where $\phi$ is the 
unique order-preserving bijection $\{w_1,w_2,\dots,w_n\} \to \{1,2,\dots,n\}$.
Now consider the 
subsets $\cA(n)  $ and $\cB(m,n) $ of $S_n$ defined recursively as follows.

Let  $\cA(n)$ be the set of permutations $\pi = \pi_1\pi_2\cdots \pi_n \in S_n$
with $\pi_1 =\pi_{n} +1$ and $\st(\pi_2\pi_3\cdots \pi_{n-1}) \in \cA(n-2)$,
where $\cA(1) =\{1\}$ and $\cA(2) = \{21\}$.
For example: 
\[ 
\ba
\cA(3) &= \{231,312\},\\
\cA(4) &= \{2431, 3412, 4213\},\\
\cA(5) &= \{24531, 25341, 34512, 35142, 42513, 45123, 52314, 53124\}.
\ea
\]
Next, let
$\cB(1,n) = \cB(n,n) = \{ n\cdots 321\}$
and define
$\cB(m,n)$ for $1<m<n$ to be the set of permutations $\pi=\pi_1\pi_2\cdots \pi_{n} \in S_n$
with either
$\pi_1=m$ and $\st(\pi_2\pi_3\cdots \pi_{n}) \in \cB(m-1,n-1)$
or $\pi_n = m$ and $\st(\pi_1\pi_2\cdots \pi_{n-1}) \in \cB(m,n-1)$. For example:
\[ 
\ba
\cB(2,3) &= \{231,312\},\\
\cB(2,4) &= \{2431, 3412, 4132\},\\
\cB(3,4) &= \{3241 , 3412 , 4213\},\\
\cB(2,5) &= \{25431 , 35412 , 45132 , 51432\},\\
\cB(3,5) &= \{32541 , 34512 , 35142 , 42513 , 45123 , 52143\},\\
\cB(4,5) &= \{43251 , 43512 , 45213 , 53214\}.
\ea
\]
The set $\cA(n)$ has size $(n-1)!! = (n-1)(n-3)(n-5)\cdots$ while $\cB(m,n)$ has size $\binom{n-1}{m-1}$.
One can show that all elements of $\cA(n)$ have length $\binom{p}{2} + \binom{q}{2}$ for $p = \lfloor \frac{n+1}{2} \rfloor$ and $q= \lceil \frac{n+1}{2} \rceil$,
while all elements of $\cB(m,n)$ have length $\binom{m}{2} + \binom{n-m+1}{2}$.

On computing $\fkR(\pi)$ for $\pi \in \cA(n)$ and $\pi \in \cB(m,n)$,
one observes the following phenomenon:

\begin{proposition}\label{prop0}
If $p \in\{ \lfloor \frac{n+1}{2} \rfloor, \lceil \frac{n+1}{2} \rceil\}$ then $ 
\sum_{\pi \in \cA(n)} \fkR(\pi)
=
 \sum_{\pi \in \cB(p,n)} \fkR(\pi)
$.
\end{proposition}

For example, it holds that
\[\ds\sum_{\pi \in \cA(5)} \fkR(\pi) = 9 + 10 + 5 + 16 + 16 + 5 + 10 + 9 = 19 + 5 + 16 + 16 + 5 + 19 =  \sum_{\pi \in \cB(3,5)} \fkR(\pi).\]
%
%
The only method we know to prove Proposition~\ref{prop0} is algebraic and indirect.
We suspect that there should exist a natural bijection $\bigsqcup_{\pi \in \cA(n)} \cR(\pi) \leftrightarrow \bigsqcup_{\pi \in \cB(p,n)} \cR(\pi)$ for
$p \in\{ \lfloor \frac{n+1}{2} \rfloor, \lceil \frac{n+1}{2} \rceil\}$, where $\cR(\pi)$ denotes the set of reduced words for $\pi \in S_n$.
Computer experiments indicate that such a bijection might be obtained from a variation of the  \emph{Little map} introduced in \cite{Little}.
It remains an open problem to construct this.

A non-bijective proof of Proposition~\ref{prop0} goes as follows.
 For integers $p,q>0$, let
 \[N(p, q) = N(q,p)= \tbinom{P+Q}{P} \cdot \fkR(p\cdots 321) \cdot \fkR(q\cdots 321)\]
 where $P = \binom{p}{2}$ and $Q=\binom{q}{2}$.
The following combines \cite[Theorem 3.7]{CanJoyce} and \cite[Theorem 1.4]{HMP1}:

\begin{proposition}[See \cite{CanJoyce,HMP1}]
If $p = \lfloor \frac{n+1}{2} \rfloor$ and $q= \lceil \frac{n+1}{2} \rceil$ then $N(p,q) = \sum_{w \in \cA(n)} \fkR(w)$.
\end{proposition}

The next statement, which is a corollary of our new results in this paper,
implies Proposition~\ref{prop0}:

\begin{proposition} \label{new-prop}
If $1 \leq p \leq n$ and $q=n+1-p$ then $N(p,q) = \sum_{w \in \cB(p,n)} \fkR(w).$
\end{proposition}

Write $\pi \overset{\bullet}= \pi'\pi''$ to denote a length-additive factorization of $\pi \in S_n$ as a product of $\pi',\pi'' \in S_n$.
To prove the last result, we interpret
the two sides of Proposition~\ref{new-prop}
as the images of a product of two elements of a certain bialgebra $\bfPi$
under a natural homomorphism $\bfPi \to \QQ$.
The following summarizes several of our main results, and gives the facts needed for this approach:

\begin{theorem}\label{intro-thm}
Let $\kk$ be a field.
There exists a graded $\kk$-bialgebra $\bfPi$ with a basis given by the symbols $[\pi]$,
where $\pi$ ranges over all elements of $S_1 \sqcup S_2 \sqcup S_3 \sqcup \cdots$,
with the following properties:
\ben
\item[(a)] The product of $[m\cdots 321]$ and $[(n-m+1)\cdots 321]$ in $\bfPi$ is $\sum_{\pi \in \cB(m,n)} [\pi]$.
\item[(b)] The coproduct of $\bfPi$ is the linear map with $[\pi] \mapsto \sum_{\pi \overset{\bullet}= \pi'\pi''} [\pi']\otimes [\pi'']$.
\item[(c)] If $\ch(\kk) =0$ then the linear map $\bfPi \to \kk$ with $[\pi] \mapsto \frac{\fkR(\pi)}{\ell(\pi)!}$ is an algebra morphism.
\een
\end{theorem}

The bialgebra $\bfPi$ is non-commutative and non-cocommutative, and 
is a sub-object of a larger bialgebra of words $\bfW$ constructed in Section~\ref{word-sect}.
The bialgebra $\bfW$ can be viewed as a generalization of
the Malvenuto-Poirier-Reutenauer Hopf algebra \cite{AguiarSottile,Mal2}.

Theorem~\ref{sub-b-thm} asserts the existence of the bialgebra $\bfPi$.
Part (a) of Theorem~\ref{intro-thm} is a special case of Theorem~\ref{a-shuff-thm}, 
which describes, more generally, the product in $\bfPi$ of any two basis elements $[\pi']$ and $[\pi'']$.
Part (b) is Corollary~\ref{coprod-cor}
and
part (c) follows from Corollary~\ref{alg-hom-cor}. 
Given these results, we can quickly derive Proposition~\ref{new-prop} in the following way:

\begin{proof}[Proof of Proposition~\ref{new-prop}]
Expressing the product in $\bfPi$ of $[p\cdots 321]$ and $[q\cdots 321]$ as in (a) and 
then applying the morphism in (c)
gives $\frac{N(p,q)}{(P+Q)!} = \sum_{\pi \in \cB(p,n)} \frac{\fkR(\pi)}{(P+Q)!}$
for $P= \binom{p}{2}$ and $Q = \binom{q}{2}$.
\end{proof}

The bialgebra of $\bfPi$ is of interest on its own, and can be used to give a simple construction 
of the \emph{Stanley symmetric function} $F_\pi$ of a permutation $\pi \in S_n$.
The precise definition of $F_\pi$ is reviewed in Section~\ref{type-a-sect}.

Let $\QSym$ (see Section~\ref{qsym-sect}) denote the Hopf algebra of quasi-symmetric functions over $\kk$
and write $\zetaq : \QSym \to \kk$ for the algebra morphism that sets $x_1=1$ and $x_2=x_3=\dots =0$.
Next let $\zeta_> : \bfPi \to \kk$ be the linear map with 
\[ \zeta_>([\pi]) = \begin{cases} 1 &\text{if $\pi$ is such that $\pi_i = i-1$ whenever $\pi_i < i$,}
\\
0&\text{otherwise}.
\end{cases}
\]
Equivalently, $\zeta_>([\pi])=1$ if and only if $\pi$ has a cycle decomposition in which every factor has the form
$(b,b-1,\cdots,a+1,a)$ for some integers $a<b$;
this occurs precisely when $\pi$ has a decreasing reduced word.
The following combines Theorem~\ref{unique-thm} and Proposition~\ref{stan-prop}:

\begin{theorem}\label{intro-thm2}
There is a unique graded bialgebra morphism $\Psi_> : \bfPi \to \QSym$ that satisfies 
$\zeta_> = \zetaq \circ \Psi_>$. For this map, $\Psi_>([\pi])=F_\pi$ is the Stanley symmetric function of $\pi \in S_n$.
\end{theorem}

This result  lets us lift enumerative identities like Proposition~\ref{new-prop}
to the level of symmetric functions; see, for example, Corollary~\ref{p-cor}.

There are also analogues of $\bfPi$ in types B/C and D,
which can be used to give a simple algebraic construction of the Stanley symmetric functions in the other classical types.
The relevant objects $\bfPi^B$ and $\bfPi^D$ are vector spaces spanned by signed permutations.
These spaces are no longer bialgebras,
but are naturally interpreted as graded $\bfPi$-module coalgebras.
It is an open problem to find an explicit formula for the $\bfPi$-module action on these spaces,
in the style of Theorem~\ref{a-shuff-thm}; see Problem~\ref{problem}.

 The Hopf algebra of quasi-symmetric functions may be 
viewed as a $\bfPi$-module coalgebra via (an analogue of) the morphism in Theorem~\ref{intro-thm2}.
In Section~\ref{last-subsect}, we show that there are canonical module coalgebra morphisms $\bfPi^B \to \QSym$ and $\bfPi^D \to \QSym$
under which the image of a signed permutation is precisely the associated Stanley symmetric function of type B, C, or D.

Here is a brief outline of what follows.
After some preliminaries in Section~\ref{prelim-sect},
we construct the bialgebras $\bfW$ and $\bfPi$ in Sections~\ref{word-sect} and \ref{perm-sect}.
In Section~\ref{qsym-sect}, we review and slightly extend some basic facts about combinatorial coalgebras and Hopf algebras from \cite{ABS}.
Finally, in Section~\ref{stan-sect}, we discuss the relationship between our constructions and Stanley symmetric functions.

\subsection*{Acknowledgements}

This work was partially supported by Hong Kong RGC Grant ECS 26305218.
I am grateful to Zachary Hamaker, Amy Pang, and Brendan Pawlowski
for many useful comments and discussions.

\section{Preliminaries}\label{prelim-sect}

Let $\ZZ \supset \NN \supset \PP$ denote the sets of all integers, nonnegative integers,
and positive integers.

\subsection{Algebras, coalgebras, and bialgebras}\label{monoidal-sect}

Throughout, we fix a field $\kk$ and write ${\otimes} = {\otimes_\kk}$ for the usual tensor product.
We briefly review the notions of $\kk$-algebras, coalgebras, and bialgebras; for more background, see \cite{Cartier} or \cite{ReinerNotes}.

\begin{definition}
A \emph{$\kk$-algebra} is a triple $(A,\nabla,\iota)$ where $A $ is a $\kk$-vector space
and $\nabla : A \otimes A \to A$ and $\iota : \kk \to A$ are linear maps (the \emph{product} and \emph{unit}) making these diagrams commute:
\[\label{assoc0-eq}
{\footnotesize
\begin{diagram}[small]
\kk\otimes  A & \rTo^{\ \ \iota\otimes\id\ \ } & A\otimes A & \lTo^{\ \ \id\otimes \iota\ \ } & A \otimes \kk \\
 & \rdTo^{\cong} & \dTo^{\nabla} &  \ldTo^{\cong}\\
 & & A 
 \end{diagram}
\qquad\qquad
\begin{diagram}[small]
 A\otimes A \otimes A & \rTo^{\ \ \nabla\otimes \id\ \ }& A \otimes A\\
\dTo^{\id \otimes \nabla}  && \dTo_{\nabla} \\
A \otimes A & \rTo^{\nabla} & A 
\end{diagram}
}
\]
\end{definition}

The unit map $\iota : \kk \to A$ of an algebra is completely determined by the \emph{unit element} $\iota(1) \in A$.

\begin{definition}
A \emph{$\kk$-coalgebra} is a triple $(A,\Delta,\epsilon)$ where $A $ is a $\kk$-vector space
and $\Delta : A \to A \otimes A$ and $\epsilon : A \to \kk$ are linear maps (the \emph{coproduct} and \emph{counit}) making these
diagrams  commute:
\[\label{coassoc0-eq}
{\footnotesize
\begin{diagram}[small]
\kk\otimes  A & \lTo^{\ \ \epsilon\otimes\id\ \ } & A\otimes A & \rTo^{\ \ \id\otimes \epsilon\ \ } & A \otimes \kk \\
 & \luTo^{\cong} & \uTo^{\Delta} &  \ruTo^{\cong}\\
 & & A 
 \end{diagram}
\qquad\qquad
\begin{diagram}[small]
 A\otimes A \otimes A & \lTo^{\ \ \Delta\otimes \id\ \ }& A \otimes A\\
\uTo^{\id \otimes \Delta}  && \uTo_{\Delta} \\
A \otimes A & \lTo^{\Delta} & A 
\end{diagram}
}
\]
\end{definition}

Write $\beta : A \otimes B \xrightarrow{\sim} B \otimes A$ for the linear isomorphism with $a\otimes b \mapsto b \otimes a$.
An algebra is \emph{commutative} 
if $\nabla\circ \beta = \nabla$.
A coalgebra is \emph{cocommutative} if $\beta\circ \Delta = \Delta$.

\begin{definition}
A \emph{$\kk$-bialgebra}  is a tuple $(A,\nabla,\iota,\Delta,\epsilon)$ where
$(A,\nabla,\iota)$ is a $\kk$-algebra, $(A,\Delta,\epsilon)$ is a $\kk$-coalgebra,
the composition $\epsilon\circ \iota$ is the identity map $\kk \to \kk$, and these diagrams commute:
\be\label{compat-eq}
{\footnotesize
 \begin{diagram}[small]
 A\otimes A&\rTo^{\nabla} & A & \rTo^{\Delta} & A \otimes A \\
\dTo^{\Delta \otimes \Delta}  & &&& \uTo_{\nabla \otimes \nabla} \\
A \otimes  A \otimes A \otimes A  && \rTo^{\id\otimes \beta \otimes \id} &&
 A \otimes A \otimes A  \otimes A
\end{diagram}
\qquad\qquad
\begin{diagram}[small]
\kk & \rTo^{\iota} & A
\\ 
\dTo^{\cong} && \dTo_{\Delta} \\
\kk\otimes \kk & \rTo^{\ \ \iota \otimes \iota\ \ } & A \otimes A 
\end{diagram}
\qquad\qquad
\begin{diagram}[small]
A \otimes A& \rTo^{\ \ \epsilon\otimes \epsilon\ \ } & \kk \otimes \kk
\\ 
\dTo^{\nabla} && \dTo_{\cong} \\
A & \rTo^{\epsilon} & \kk
\end{diagram}
}
\ee
\end{definition}

Going forward, we often refer to $\kk$-vector spaces, $\kk$-algebras, $\kk$-coalgebras, and $\kk$-bialgebras
simply as vector spaces, algebras, coalgebras, and bialgebras.

A morphism of (bi, co) algebras is a linear map that commutes with the relevant (co)unit and (co)product maps.
If $A$ is an algebra then $A \otimes A$ is an algebra with product $(\nabla \otimes \nabla) \circ (\id \otimes \beta \otimes \id)$ and unit $(\iota \otimes \iota) \circ (\kk \xrightarrow{\sim} \kk \otimes \kk)$.
If $A$ is a coalgebra then $A\otimes A$ becomes a coalgebra in a similar way.
The diagrams \eqref{compat-eq} express that the coproduct and counit of a bialgebra are algebra morphisms,
and that the product and unit are coalgebra morphisms.

Given a bialgebra $(H,\nabla,\iota,\Delta,\epsilon)$ and linear maps $f,g : H \to H$,
define $f* g = \nabla \circ (f\otimes g) \circ \Delta$.
The operation $* $, called the \emph{convolution product}, 
makes the vector space $\End(H)$ of linear maps $H \to H$ into a $\kk$-algebra with unit element $\iota \circ \epsilon$,
called the \emph{convolution algebra} of $H$.
The bialgebra $H$ is a \emph{Hopf algebra} if the identity map $\id : H \to H$ has a left and right inverse $\antipode : H \to H$ in the convolution algebra.
The morphism $\antipode$ is called the \emph{antipode} of $H$; if it exists, then 
it is the unique linear map $H \to H$ with
 $\nabla \circ (\id \otimes \antipode) \circ \Delta=\nabla \circ (\antipode \otimes \id) \circ \Delta=\iota\circ \epsilon$.

A vector space $V$ is \emph{graded} if it is has a direct sum decomposition $V = \bigoplus_{n\in \NN} V_n$.
A linear map $\phi : U \to V$ between graded vector spaces is \emph{graded} if it has 
the form $\phi = \bigoplus_{n \in \NN} \phi_n$ where each $\phi_n: U_n \to V_n$ is linear.
An algebra $(V,\nabla,\iota)$ is \emph{graded} if $V$ is graded and the product and unit are graded linear maps,
where $V\otimes V$ is identified with the graded vector space $\bigoplus_{n \in \NN} \( \bigoplus_{i+j = n} V_i \otimes V_j\)$
and $\kk$ is viewed as the graded vector space in which all elements have degree zero.
Similarly, a coalgebra $(V,\Delta,\epsilon)$ is \emph{graded} if $V$ is graded and 
the coproduct and counit are graded linear maps.
A bialgebra is \emph{graded} if it is graded as both an algebra and a coalgebra.
A morphism of graded (bi, co) algebras is a morphism of (bi, co) algebras that is a graded linear map.

\subsection{Shuffle algebra}\label{shuffle-sect}

Throughout, we use the term \emph{word} to mean a finite sequence of positive integers.
For each $n \in \NN$,
let $\Shuffle_n$ be the $\kk$-vector space whose basis is the set of all $n$-letter words,
so that $\Shuffle_0$ is the
1-dimensional vector space spanned by the unique empty word $\emptyset$.
Write $\Shuffle = \bigoplus_{n \in \NN} \Shuffle_n$ for the corresponding graded vector space.

When $w =w_1w_2\cdots w_n$ is a word and $I=\{i_1 < i_2 < \dots < i_k\} \subset \{1,2,\dots,n\}$, we
set $w|_I = w_{i_1}w_{i_2}\cdots w_{i_k}$.
Given words $u= u_1u_2\cdots u_m$ and $v=v_1v_2\cdots v_n$,
define \[u \shuffle v = \sum_{\substack{ I\subset \{1,2,\dots,m+n\} \\ |I| = m}} \shuffle_I(u,v)\] where $w=\shuffle_I(u,v)$ is the unique $(m+n)$-letter word with $w|_I = u$ and $w|_{I^c} = v$.
Multiplicities may result in this sum; for example,  
$
12\shuffle 21 = 2\cdot 1221  + 1212 + 2121 + 2 \cdot 2112.
$
Let $u\odot v $ denote the concatenation of $u$ and $v$.
Both operations $\shuffle$ and $\odot$ extend to graded linear maps $\Shuffle \otimes \Shuffle \to \Shuffle$.
If $u$ and $v$ are nonempty, and $u'$ and $v'$ are the subwords given by omitting the first letters, 
then 
$u\shuffle v = u_1\odot (u'\shuffle v) + v_1\odot (u\shuffle v').$
We typically suppress the symbol $\odot$ and write $uv$ for $u\odot v$.

Define $\iota : \kk \to \Shuffle$ and $\epsilon : \Shuffle \to \kk$
to be the linear maps with $\iota(1) = \emptyset$
and $\epsilon(\emptyset) = 1$ 
and $\epsilon(w) = 0$ for all words $w\neq \emptyset$.
Define $\Delta : \Shuffle \to \Shuffle \otimes \Shuffle$ to be the 
linear map with
\[\Delta(w) = \sum_{i=0}^n w_1\cdots w_i \otimes w_{i+1}\cdots w_n\] for each $n$-letter word $w$, so that $\Delta(\emptyset) = \emptyset\otimes \emptyset$.
The tuple $(\Shuffle,\shuffle,\iota,\Delta,\epsilon)$ is a graded Hopf algebra,
 called the \emph{shuffle algebra}  \cite[\S1.4]{FreeLieBook}, which is commutative but not cocommutative.
Its antipode is the linear map   
with $\antipode(w) = (-1)^n w_n \cdots w_2w_1$ for words $w=w_1w_2\cdots w_n$.

\section{Bialgebras of words}\label{word-sect}

Let $\End(\Shuffle)$ denote the vector space of $\kk$-linear maps $\Shuffle \to \Shuffle$,
viewed as a $\kk$-algebra with respect to the convolution product
$f\otimes g \mapsto f * g := \shuffle \circ (f\otimes g) \circ \Delta$.
One can sometimes construct interesting bialgebras by pairing 
subalgebras of convolution algebras with a compatible coproduct. This is our approach here,
mimicking the description of the Malvenuto-Poirier-Reutenauer algebra in \cite{Hazewinkel1,Mal2}.

Let $w = w_1w_2\cdots w_m$ be a word. 
When $m>0$, let $\max(w)=\max \{ w_1,w_2,\dots,w_m\}$,
and define $\max(\emptyset) = 0$.
If we view $w$ as a map $\{1,2,\dots,m\} \to \PP$,
and if $v$ is a word with at least $\max(w)$ letters, then
the composition $v\circ w = v_{w_1} v_{w_2}\cdots v_{w_m}$ is another $m$-letter word.
We interpret  $v \circ \emptyset$ as $\emptyset$.
For each word $w$ with $\max(w) \leq n \in \NN$, define $[w,n] \in \End(\Shuffle)$ 
to be the linear map with
\[ [w,n](v) = \begin{cases} v\circ w=v_{w_1} v_{w_2}\cdots v_{w_m} & \text{if $v$ is a word with length $n$} \\ 
0&\text{for all other words $v$.}
\end{cases}
\]
Observe that $[\emptyset, 0]= \iota \circ \epsilon$ is the unit element of $\End(\Shuffle)$. The following is evident:

\begin{lemma}\label{obv-lem}
If  $w$ is a word with $\max(w) \leq n \in \NN$, then $[w,n](123\cdots n)= w$.
\end{lemma}

Thus $[v,m] = [w,n]$ in $\End(\Shuffle)$ if and only if $v=w$ and $m=n$.
Let $\Words_n$ for $n \in \NN$ be the set of endomorphisms $[w,n]$ where $w$ is a word with $\max(w) \leq n$.
Define $\Words = \bigcup_{n \in \NN} \Words_n$.
Let $\bfW$ (respectively, $\bfW_n$)  be the subspace of $\End(\Shuffle)$ spanned by $\Words$ (respectively, $\Words_n$).
Since the set of words is a basis for $\Shuffle$, Lemma~\ref{obv-lem} implies that 
$\Words$ is linearly independent. Therefore:

\begin{corollary}\label{basis-cor}
The set $\Words$ is a basis for $\bfW=\bigoplus_{n \in \NN} \bfW_n$.
\end{corollary}

When $w=w_1w_2\cdots w_m$ is a word and $n \in \NN$,
define $w\uparrow n=(w_1+n)(w_2+n)\dots (w_m+n)$
to be the word formed by incrementing each letter of $w$ by $n$.
When $w^1,w^2,\dots, w^l$ is a finite sequence of words with $\max(w^i) \leq n$ and $a_1,a_2,\dots,a_l \in \kk$, 
let $\left[ \sum_{i } a_i w^i, n \right] = \sum_{i} a_i [w^i,n] \in \bfW_n$.
Finally, define $\nabla_\shuffle : \bfW \otimes \bfW \to \bfW $ to be the linear map with
\be\label{nabla-eq}
\nabla_\shuffle([v,m]\otimes [w,n]) =  [v \shuffle (w\uparrow m), n+m] \in \bfW_{m+n}
\ee
for $[v,m] \in \Words_m$ and $[w,n] \in \Words_n$.
For example,
$\nabla_\shuffle([12,3]\otimes [2,2]) = [125,5] + [152,5] +  [512,5]$.
Note that $v \shuffle (w\uparrow m)$ is the multiplicity-free sum of all words $u$ with $u \cap \{1,2,\dots,m\} = v$
and $u \cap (m + \PP) = w \uparrow m$, where $u\cap S$ is the subword formed by removing all
letters not in $S$.

\begin{proposition}
If $\alpha,\beta \in \bfW$ then
$\nabla_\shuffle(\alpha\otimes\beta) = \alpha * \beta$.
\end{proposition}

Thus, $\nabla_\shuffle$ is associative and $\bfW$ is a subalgebra of the convolution algebra $\End(\Shuffle)$.

\begin{proof}
Let $[v,m] \in \Words_m$ and $[w,n] \in \Words_n$.
If $u=u_1u_2\cdots u_l$ is a word then $([v,m] * [w,n])(u) $
is the sum over $i \in \{0,1,\dots,l\}$ of $[v,m](u_1u_2\cdots u_i) \shuffle [w,n](u_{i+1}u_{i+2}\cdots u_l)$,
which is precisely
$(u\circ v) \shuffle (u \circ (w\uparrow m))= [v \shuffle (w\uparrow m)), n+m] (u)$
if $l=m+n$ and zero otherwise.
\end{proof}

%

The coalgebra structure of $\Shuffle$ induces a coalgebra structure on $\bfW$.
For $n \in \NN$, let $\rho_n : \Shuffle \to \bfW_n$ be the surjective linear map with $\rho_n(w) = [w,n]$ if $w$ is a word with $\max(w) \leq n$ and with $\rho_n(w) = 0$ if $\max(w) > n$.
Write $\epsilon $ and $\Delta$ for the counit and coproduct of $\Shuffle$.

\begin{lemma}\label{coideal-lem}
For each $n \in \NN$, it holds that $\ker (\rho_n) \subset \ker (\epsilon) \cap \ker\( (\rho_n \otimes \rho_n) \circ \Delta\)$.
\end{lemma}

\begin{proof}
A basis for $\ker \rho_n$ is the set of words $w \in \Shuffle$ with $\max(w) > n$.
Such a word $w=w_1w_2\cdots w_m$ is nonempty and  
$\max(w_1\cdots w_i)> n$ or $\max(w_{i+1}\cdots w_m)> n$ for each $0\leq i \leq m$. 
\end{proof}

It follows that there are unique linear maps $\epsilon_\odot : \bfW \to \kk$ and $\Delta_\odot : \bfW \to \bfW \otimes \bfW$
satisfying 
$\epsilon_\odot(\rho_n(w)) = \epsilon(w)$ and $ \Delta_\odot(\rho_n(w)) = (\rho_n \otimes \rho_n)(\Delta(w))$
for each word $w $ and integer $n \in \NN$.
If $[w,n] \in \Words_n$ and $w=w_1w_2\cdots w_m$, then these maps have the explicit formulas
\be
\label{odot-maps}
 \epsilon_\odot([w,n]) = \begin{cases} 
1 &\text{if } w = \emptyset \\
0 &\text{otherwise}
\end{cases}
\quand
\Delta_{\odot}([w,n]) = \sum_{i=0}^m [w_1\cdots w_i, n] \otimes [w_{i+1}\cdots w_m,n].
\ee
Write $\iota_\shuffle $ for the linear map $\kk \to \bfW$ with 
$
\iota_\shuffle(1) = [\emptyset, 0].
$
We consider $\bfW$ to be a graded vector space in which $[w,n] \in \Words_n$ is homogeneous with degree $\ell(w)$, the length of the word $w$.

\begin{theorem}\label{w-thm}
$(\bfW, \nabla_\shuffle, \iota_\shuffle, \Delta_\odot, \epsilon_\odot)$ is a graded bialgebra, but not a Hopf algebra.
\end{theorem}

\begin{proof}
 Lemma~\ref{obv-lem} implies that $(\bfW, \nabla_\shuffle, \iota_\shuffle)$ is a graded algebra.
For each $n \in \NN$,
the subspace $\ker \rho_n$ is a coideal of
$(\Shuffle,\Delta,\epsilon)$ by Lemma~\ref{coideal-lem}, and 
 $(\bfW_n,\Delta_\odot,\epsilon_\odot)$ is the graded coalgebra obtained by transferring the structure maps of the quotient
$\Shuffle / \ker (\rho_n)$
via the map $\rho_n$. 
Direct sums of coalgebras are coalgebras, so $(\bfW,\Delta_\odot,\epsilon_\odot)$ is a graded coalgebra.

We have $\epsilon_\odot \circ \iota_\shuffle = \id$,
and 
the counit (respectively, unit) is obviously an algebra (respectively, coalgebra) morphism.
Let $\beta$ denote the standard isomorphism $U\otimes V \xrightarrow{\sim} V\otimes U$
 and write $\nabla$ for the product $v\otimes w \mapsto v\shuffle w$ of $\Shuffle$. 
Let $[v,m],[w,n] \in \Words$. It follows from the definitions that
\[\label{compat1} 
\ba
\Delta_\odot \circ \nabla_\shuffle([v,m]\otimes [w,n]) &=
 \Delta_\odot \circ \rho_{m+n} \circ \nabla(v\otimes (w\uparrow m)) =
(\rho_{m+n}) ^{\otimes2} \circ \Delta\circ \nabla( v \otimes (w\uparrow m)).
\ea
\]
Similarly, we have
\[\ba
(\nabla_\shuffle)^{\otimes2} \circ (\id \otimes \beta \otimes \id)\circ  &(\Delta_\odot )^{\otimes2} ([v,m] \otimes [w,n])
\\&=(\nabla_\shuffle)^{\otimes2}  \circ (\id \otimes \beta \otimes \id)\circ (\rho_{m} \otimes \rho_m\otimes \rho_{n} \otimes \rho_n) \circ \Delta^{\otimes 2}(v \otimes w)
\\&
=
(\rho_{m+n})^{\otimes2} \circ \nabla^{\otimes2}\circ (\id \otimes \beta \otimes \id)\circ \Delta^{\otimes2} ( v \otimes (w \uparrow m)).
\ea
\]
It holds that $\Delta\circ \nabla = \nabla^{\otimes2}\circ (\id \otimes \beta \otimes \id)\circ \Delta^{\otimes2} $
since $\Shuffle$ is a Hopf algebra, so the final expressions in the two equations are equal.
Thus $\Delta$ is an algebra morphism so
  $(\bfW, \nabla_\shuffle, \iota_\shuffle, \Delta_\odot, \epsilon_\odot)$ is a graded bialgebra.

This bialgebra is not a Hopf algebra since if
 $\antipode : \bfW \to \bfW$ is a linear map and $n \in \PP$, then
$\nabla_\shuffle \circ (\id \otimes \antipode) \circ \Delta_\odot([\emptyset,n])
= \nabla_{\shuffle}([\emptyset,n]\otimes \antipode([\emptyset,n]))\in \bigoplus_{m \in \NN} \bfW_{m+n}$,
which means that
\[\nabla_\shuffle \circ (\id \otimes \antipode) \circ \Delta_\odot([\emptyset,n])\neq \iota_\shuffle \circ \epsilon_\odot ([\emptyset,n]) = [\emptyset,0] \in \bfW_0.\]
The map  $ \iota_\shuffle\circ \epsilon_\odot  $
is therefore not invertible in the convolution algebra $\End(\bfW)$.
\end{proof}

\begin{corollary}\label{alg-hom-cor}
When $\ch(\kk) = 0$, the linear map $ \bfW \to \kk$ with $[w,n]\mapsto \frac{1}{\ell(w)!}$ for $[w,n] \in \Words$ is an algebra morphism.
\end{corollary}

\begin{proof}
If $\lambda$ is this map, then $\lambda\circ\nabla_\shuffle([v,m]\otimes[w,n]) = \binom{\ell(v)+\ell(w)}{\ell(v)} \frac{1}{(\ell(v) + \ell(w))!} = \lambda([v,m])\lambda([w,n])$.
\end{proof}

The bialgebra $\bfW$ has a well-known quotient.
Let $w=w_1w_2\cdots w_n$ be a word. Suppose the set $S = \{w_1,w_2,\dots,w_n\}$ 
has $m$ distinct elements. 
If $\phi $ is the unique order-preserving bijection $S \to \{1,2,\dots,m\}$,
then we define
 $\st(w) = \phi(w_1)\phi(w_2)\cdots \phi(w_n)$.

A \emph{packed word} is a word $w$ with $w = \st(w)$.
Such a word is just a surjective map $\{1,2,\dots,n\}\to\{1,2,\dots,m\}$ for some $m,n \in \NN$.
Packed words are also referred to in the literature as
\emph{surjective words} \cite{Hazewinkel1}, \emph{Fubini words} \cite{Brends}, and \emph{initial words} \cite{PylPat}. 
Define $\bfI_\P$ to be the subspace of $\bfW$ spanned by all differences $[v,m] - [w,n]$ 
where $[v,m],[w,n] \in \Words$ have $\st(v) = \st(w)$.

\begin{proposition}\label{packed-prop}
The subspace $\bfI_\P$ is a homogeneous bi-ideal of $(\bfW, \nabla_\shuffle, \iota_\shuffle, \Delta_\odot, \epsilon_\odot)$.
The quotient bialgebra $\bfP = \bfW /\bfI_\P$ is a graded Hopf algebra.
\end{proposition}

$\bfP $ is 
the graded dual of the Hopf algebra of \emph{word quasi-symmetric functions} $\WQSym$ \cite{NovelliThibon}.

\begin{proof}
The subspace $\bfI_\P$ is homogeneous since two words $v$ and $w$ must have the same length if 
$\st(v) = \st(w)$. It is easy to check that $\bfI_\P$ is a bi-ideal.
The quotient bialgebra $\bfP = \bfW /\bfI_\P$ is connected and therefore a graded Hopf algebra \cite[\S2.3.2]{AguiarMahajan}.
\end{proof}

Let $\Packed$ be the set of all packed words. 
If $[w,n] \in \Words$ and $w$ is a word with $m$ distinct letters then $v = \st(w)$
is the unique packed word such that $[w,n] + \bfI_\P = [v,m] + \bfI_\P$.
Identify $v \in \Packed$
with the coset $[v,m] + \bfI_\P$ so that we can view $\Packed$ as a basis for $\bfP$.
It is not hard to work out formulas for the product and coproduct of $\bfP$ in the basis $\Packed$, but we omit these details here.
The subspace of $\bfP$ spanned by the 
words in $\Packed$ that have no repeated letters
is a Hopf subalgebra, namely, the \emph{Malvenuto-Poirier-Reutenaurer Hopf algebra} of permutations \cite{AguiarSottile,Mal2},
sometimes also called the Hopf algebra of \emph{free quasi-symmetric functions} $\FQSym$ \cite{DHT}.

\section{Bialgebras of permutations}\label{perm-sect}

Recall that $S_n$ denotes the group of permutations of $\{1,2,\dots,n\}$.
For each $1\leq i \leq n-1$, let $s_i \in S_n$ be the simple transposition given in cycle notation by $(i,i+1)$.
Then $S_n$ is the finite Coxeter group of type $A_{n-1}$
relative to the generating set $\{s_1,s_2,\dots,s_{n-1}\}$.

A \emph{reduced word} for $\pi \in S_n$ is a word $i_1i_2\cdots i_l$
of minimal length such that $\pi = s_{i_1}s_{i_2}\cdots s_{i_l}$.
Let $\cR(\pi)$ be the set of such words
and define $\ell(\pi)$ to be their common length.
For $n \in \NN$ and $\pi \in S_{n+1}$, define 
$[\pi] = \sum_{w \in \cR(\pi)} [w, n] \in \bfW_n$.
Let $\bfPi_n = \kk\spanning \{ [\pi] : \pi \in S_{n+1}\}$ and $\bfPi = \bigoplus_{n \in \NN} \bfPi_n$.

\begin{theorem}\label{sub-b-thm}
The subspace $\bfPi$ is a graded sub-bialgebra of $(\bfW,\nabla_\shuffle,\iota_\shuffle,\Delta_\odot,\epsilon_\odot)$.
\end{theorem}

This is a special case of more general results in \cite[\S5]{M2}. We include a self-contained proof.

\begin{proof}
The \emph{word property} for Coxeter groups \cite[Theorem 3.3.1]{CCG} asserts that
for each $\pi \in S_n$,  the set $\cR(\pi)$ is  an equivalence class
under
the strongest relation with
$vw\sim v'w'$ whenever $v\sim v'$ and $w\sim w'$,
such that 
$ij  \sim  ji $
and
$  i(i+1)i \sim  (i+1)i(i+1) 
$
for $i,j \in \PP$ with $|j-i| > 1$.
Moreover, an equivalence class under this relation is equal to $\cR(\pi)$ for some permutation $\pi$
if and only if it contains no words with adjacent repeated letters.
It is clear from these observations that the coproduct $\Delta_\odot$
satisfies $\Delta_\odot(\bfPi) \subset \bfPi \otimes \bfPi$ and has the formula in Corollary~\ref{coprod-cor}.

The unit element of $\bfW$ is $[\emptyset,0] = [\pi]$ for $\pi = 1 \in S_1$.
To show that $\nabla_\shuffle(\bfPi\otimes \bfPi) \subset \bfPi$,
let $\pi' \in S_{m+1}$ and $\pi'' \in S_{n+1}$. No word formed by 
shuffling $v$ and $w\uparrow m$ for  $v \in \cR(\pi')$
and $w \in \cR(\pi'')$
contains $m(m+1)m$ or $(m+1)m(m+1)$ as a consecutive subword,
since this would imply that two adjacent letters of $v$ or $w$ are equal.
It follows that the set of all words formed by shuffling $v$ and $w\uparrow m$ 
for some 
$v \in \cR(\pi')$
and $w \in \cR(\pi'')$
is a union of $\sim$-equivalence classes. No words $u$ in this set may contain adjacent repeated letters 
since $u \cap \{1,2,\dots,m\} \in \cR(\pi')$ and $u \cap (m + \PP) \in \cR(\pi'')$,
so $\nabla_\shuffle([\pi']\otimes[\pi'']) \in \bfPi$.
Each $[\pi]$ is homogeneous of degree $\ell(\pi)$, so $\bfPi$ is a graded sub-bialgebra.
\end{proof}

If $\pi \in S_n$ then $\ell(\pi)$ 
may be computed as the number of integer pairs $(i,j)$ with $1\leq i<j \leq n$ and $\pi(i) > \pi(j)$.
We write $\pi \overset{\bullet}= \pi'\pi''$ if $\pi,\pi',\pi'' \in S_n$ and $\pi = \pi'\pi''$
and $\ell(\pi) = \ell(\pi') + \ell(\pi'')$.

\begin{corollary}\label{coprod-cor}
If $\pi \in S_n$ then 
$\ds
\Delta_{\odot}([\pi]) = \sum_{\pi \overset{\bullet}=\pi'\pi''} [\pi']\otimes [\pi'']
$
and
$\epsilon_\odot([\pi]) = \begin{cases} 1 & \text{if }\ell(\pi)=0 \\ 0 & \text{if }\ell(\pi) \neq 0. \end{cases}$
\end{corollary}

The product of $\bfPi$ takes more work to describe.
In the following definitions,
we consider $S_m$ and $S_n$ to be disjoint for all $m\neq n$.
We 
represent elements of  $S_n$ in \emph{one-line notation}, that is, by writing
 the word
$\pi_1\pi_2\cdots \pi_n$ to mean the permutation $\pi \in S_n$ with $\pi(i) = \pi_i$.

\begin{definition}
Suppose $a=a_1a_2\cdots a_k$ and $b=b_1b_2\cdots b_l$ are words with no repeated letters.
Let $A = \{a_1,a_2,\dots,a_k\}$ and $B = \{b_1,b_2,\dots,b_l\}$ and $n=k+l$,
and assume $A$ and $B$ are both contained in $\{1,2,\dots,n\}$. 
Write $\phi$ and $\psi$ for the order-preserving bijections
$A \to \{1,2,\dots,n\}- B$ and $B \to \{1,2,\dots,n\} - A$. Finally, define
$a \dfslash b = a \psi(b) \in S_n$
and
$a\dbslash b = \phi(a) b \in S_n$.
\end{definition}

For example, $a\dfslash \emptyset = \emptyset\dfslash a = a$ and
$1357\dfslash 6543= 13578642
$
and
$1357 \dbslash 6543 = 12786543.
$

\begin{definition}\label{a-shuffle-def}
Let $u \in S_{m+1}$ and $v \in S_{n+1}$. Set $w_i = v_i +m$ 
and suppose $j$ and $k$
are the indices with $u_j = w_{k+1}= m+1$.
Define a subset 
$\cS_\shuffle(u,v) \subset S_{m+n+1}$ 
inductively as follows:
\begin{itemize}
\item[(a)] If $u_{m+1} = m+1$ then let $\cS_\shuffle(u,v) = \{ u_1u_2\cdots u_{m} w_1w_2\cdots w_{n+1}\}$.

\item[(b)] If $w_1 = m+1$ then let $\cS_\shuffle(u,v) = \{ u_1u_2\cdots u_mu_{m+1} w_2\cdots w_{n+1}\}.$

\item[(c)] Otherwise,
define
$\tilde u = \st(u_{j+1}u_{j+2}\cdots u_{m+1}) $ and
$\tilde v = \st(v_1v_2\cdots v_{k})$
and let
\[
\cS_\shuffle(u,v) = \Bigl\{ 
u_1u_2\cdots u_j \dfslash \sigma : \sigma \in \cS_\shuffle(\tilde u,v)\Bigr\}
\sqcup
\Bigl\{ \sigma \dbslash w_{k+1}w_{k+2}\cdots w_{n+1} : \sigma \in \cS_\shuffle(u,\tilde v)\Bigr\}.
\]
\end{itemize}
\end{definition}

\begin{example}
If $u = 231$ and $v=312$, then $w=534$, $j=k+1=2$,
and $\tilde u = \tilde v = 1$,
so $\cS_\shuffle(\tilde u, v)=  \{312\}$ and $\cS_\shuffle(u,\tilde v) =  \{231\}$
and \[\cS_\shuffle(231,312) = \{ 23 \dfslash 312\} \sqcup \{ 231 \dbslash  34 \} = \{23514, 25134\}.\]
On the other hand, if $u=312$ and $v=231$, then $w=453$, $j=1$ and $k+1=3$,
and $\tilde u = \tilde v = 12$,
so $\cS_\shuffle(\tilde u, v)=  \{1342\}$ and $\cS_\shuffle(u,\tilde v) =  \{3124\}$
and
\[ \cS_\shuffle(312,231) = \{ 3 \dfslash 1342\} \sqcup \{ 3124 \dbslash  3 \} =  \{31452, 41253\}.\]
For a more complicated example, one can check that
\[
\ba
\cS_\shuffle(4213,4132) &= \{ 4217365 , 7213465\}, \\
\cS_\shuffle(4132 ,4213 ) &= 
\{4137526,4157236 , 4172536 , 5137246 ,5172346 , 7132546\}.
\ea
\]
\end{example}

If $p,q \in \PP$ are such that $p+q = n+1$ then
$\cS_\shuffle(p\cdots 321, q\cdots 321)$ is just the set $\cB(p, n)$ defined in the introduction.
The following therefore implies Theorem~\ref{intro-thm}(a):

\begin{theorem}\label{a-shuff-thm}
If $u \in S_{m+1}$ and $v \in S_{n+1}$ then $\nabla_{\shuffle}([u]\otimes [v]) = \sum_{\pi \in \cS_\shuffle(u,v)} [\pi]$.
\end{theorem}

For the proof, we will need some facts about wiring diagrams.
Fix $n,N \in\PP$ and let $I = \{ x \in \RR: 0 \leq x \leq N\}$.
Let  $f_1,f_2,\dots,f_n : I \to \RR$
be continuous functions.
The tuple $D = (f_1,f_2,\dots,f_n)$ is a \emph{wiring diagram}
if $f_1(0) < f_2(0) < \dots <f_n(0)$ and for each integer $1\leq i \leq N$, we have:
\ben
\item[(1)] The numbers $f_1(i), f_2(i), \dots, f_n(i)$ are all distinct.
\item[(2)] For each  $1 \leq j \leq n$, the restriction of $f_j$ to $(i-1,i) := \{ x \in \RR : i-1<x<i\}$ is a line.
\item[(3)] At most two functions $f_j$ and $f_k$ intersect in the open interval $(i-1,i)$,
and  an intersection occurs only if $f_j(i)$ and $f_k(i)$ are consecutive elements of $\{f_1(i),f_2(i), \dots, f_n(i)\}$.
\een
Properties (1)-(3) imply that if $f_j$ and $f_k$ intersect in the open interval $(i-1,i)$, then $f_j(i-1)$ and $f_k(i-1)$ are consecutive elements in 
$\{f_1(i-1),f_2(i-1), \dots, f_n(i-1)\}$.
The function $f_j$ is the $j$th \emph{wire} of $D$.
We say that the wires $f_j$ and $f_k$ cross if $f_j(x) = f_k(x)$ for some $x \in I$.
For example, 
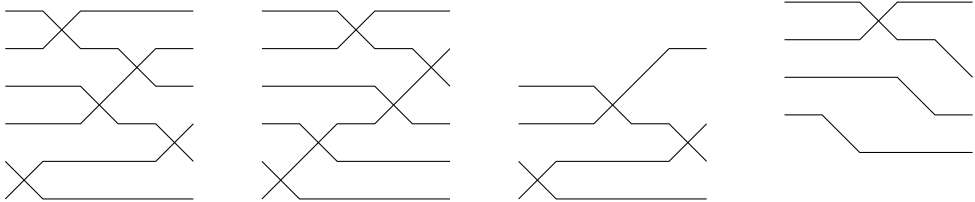
\begin{center}
\begin{tikzpicture}[scale=0.5]
\draw (0,0) -- (1,1) -- (2,1) -- (3,1) -- (4,1) -- (5,2);
\draw (0,1) -- (1,0) -- (2,0) -- (3,0) -- (4,0) -- (5,0);
\draw (0,2) -- (1,2) -- (2,2) -- (3,3) -- (4,4) -- (5,4);
\draw (0,3) -- (1,3) -- (2,3) -- (3,2) -- (4,2) -- (5,1);
\draw (0,4) -- (1,4) -- (2,5) -- (3,5) -- (4,5) -- (5,5);
\draw (0,5) -- (1,5) -- (2,4) -- (3,4) -- (4,3) -- (5,3);
\end{tikzpicture}
\qquad
\begin{tikzpicture}[scale=0.5]
\draw (0,0) -- (1,1) -- (2,2) -- (3,2) -- (4,3) -- (5,4);
\draw (0,1) -- (1,0) -- (2,0) -- (3,0) -- (4,0) -- (5,0);
\draw (0,2) -- (1,2) -- (2,1) -- (3,1) -- (4,1) -- (5,1);
\draw (0,3) -- (1,3) -- (2,3) -- (3,3) -- (4,2) -- (5,2);
\draw (0,4) -- (1,4) -- (2,4) -- (3,5) -- (4,5) -- (5,5);
\draw (0,5) -- (1,5) -- (2,5) -- (3,4) -- (4,4) -- (5,3);
\end{tikzpicture}
\qquad
\begin{tikzpicture}[scale=0.5]
\draw (0,0) -- (1,1) -- (2,1) -- (3,1) -- (4,1) -- (5,2);
\draw (0,1) -- (1,0) -- (2,0) -- (3,0) -- (4,0) -- (5,0);
\draw (0,2) -- (1,2) -- (2,2) -- (3,3) -- (4,4) -- (5,4);
\draw (0,3) -- (1,3) -- (2,3) -- (3,2) -- (4,2) -- (5,1);
\end{tikzpicture}
\qquad
\begin{tikzpicture}[scale=0.5]
\node at (0,0) {\ };
\draw (0,2) -- (1,2) -- (2,1) -- (3,1) -- (4,1) -- (5,1);
\draw (0,3) -- (1,3) -- (2,3) -- (3,3) -- (4,2) -- (5,2);
\draw (0,4) -- (1,4) -- (2,4) -- (3,5) -- (4,5) -- (5,5);
\draw (0,5) -- (1,5) -- (2,5) -- (3,4) -- (4,4) -- (5,3);
\end{tikzpicture}
\end{center}
are wiring diagrams with $N=5$ and $n=6,6,4,4$, respectively.

The wiring diagram $D$ is \emph{reduced} if all distinct wires $f_j$ and $f_k$ cross at most once.
For $0\leq i \leq N$, define $\pi^i  = \st(f_1(i)f_2(i)\cdots f_n(i)) \in S_n$
and say that $D$ is a wiring diagram for $\pi = \pi^N$.
Then $\pi^0= 1 \in S_n$, and if $0<i\leq N$ then either $\pi^i = \pi^{i-1}$ 
or $\pi^i = s_j\pi^{i-1}$ for some $j $.
Consider the word $w$ of length $N$
whose $i$th letter is $0$ if  $\pi^i = \pi^{i-1}$
or otherwise  the index $j $ such that $\pi^i = s_j\pi^{i-1}$.
Define $\word(D)$ to be the sequence formed by 
removing all zeros from $w$ and then reversing the resulting subword. 
We refer to $\word(D)$ as the \emph{word} associated to $D$.
If $i_l\cdots i_2i_1 = \word(D)$ then $\pi = s_{i_l}\cdots s_{i_2}s_{i_1}$,
and it holds that $\word(D) \in \cR(\pi)$ if and only if $D$ is reduced.
The examples above are reduced wiring diagrams for
$\pi = 315264$, $512364$, $3142 = \st(3152)$, and $1243=\st(2364)$ with words $24351$, $43521$, $231$, and $3$.

Any word $i_l\cdots i_2i_1$ such that $\pi = s_{i_l}\cdots s_{i_2} s_{i_1} \in S_n$
is the word of the wiring diagram for $\pi$ given as follows:
take $N = l$, define $\pi^0 = 1 \in S_n$ and $\pi^j =s_{i_j} \cdots s_{i_2}s_{i_1} \in S_n$, and let
$D=(f_1,f_2,\dots,f_n)$ where
$f_i$ is the piecewise linear function $I \to \RR$ connecting the points $(x,y)=(j,\pi^j(i))$ for $j=0,1,\dots,l$.
We refer to this as the \emph{standard wiring diagram} of the word $i_l\cdots i_2i_1$ .
The first two examples above are the standard wiring diagrams for $24351$ and $43521$.

Suppose $D=(f_1,f_2,\dots,f_n)$ is a wiring diagram for $\pi \in S_n$. Let $i \in \{1,2,\dots,n\}$.
Then $E = (f_1,\dots,f_{i-1},f_{i+1},\dots,f_n)$ is a wiring diagram for $\st(\pi_1\cdots \pi_{i-1}\pi_{i+1}\cdots \pi_n) \in S_{n-1}$,
and if $D$ is reduced then $E$ is also reduced.
In turn,  $(f_{i+1},f_{i+2},\dots,f_n)$ 
and $(f_{1},f_{2},\dots,f_{i-1})$ are wiring diagrams for some  $x \in S_{n-i}$
and $y \in S_{i-1}$,
and we have $\pi = \pi_1\pi_2\cdots \pi_i \dfslash x = y \dbslash \pi_i\pi_{i+1}\cdots \pi_n$.
This is illustrated by the examples above, since $315264 = 3142 \dbslash 64$ and $512364 = 51 \dfslash 1243$.

\begin{proof}[Proof of Theorem~\ref{a-shuff-thm}]
The reader may find it helpful to consult the example following the proof, which works through our argument in a concrete case.

Let $u \in S_{m+1}$, $v \in S_{n+1}$,
and  $w_i = v_i + m$.
If $u_{m+1}=m+1$ then no reduced word for $u$ involves the letter $m$,
so every word obtained by shuffling $a \in \cR(u)$ and $b\uparrow m$ for $b\in \cR(v)$
is equivalent under the Coxeter relation $\sim$ described in the proof of Theorem~\ref{sub-b-thm} to $a(b\uparrow m)$,
which is a reduced word
for $\pi = u_1u_2\cdots u_m w_1w_2\cdots w_{n+1} \in S_{m+n+1}$.
Since $\nabla_{\shuffle}([u]\otimes [v])$ is a multiplicity-free sum of elements $[\sigma]$ with $\sigma \in S_{m+n+1}$,
it follows that $\nabla_{\shuffle}([u]\otimes [v]) = [\pi]$ as claimed.
If $w_1=m+1$ then no reduced word for $v$ involves the letter $1$, 
and it follows similarly that $\nabla_{\shuffle}([u]\otimes [v]) = [\pi]$ for $\pi = u_1u_2\cdots u_{m+1} w_2w_3\cdots w_{n+1} \in S_{m+n+1}$.

Suppose $u_j = w_{k+1} = m+1$
where $1\leq j<m+1$ and $0< k \leq n$.
Then every word $a \in \cR(u)$ contains the letter $m$ 
and every word $b \in \cR(v)$ contains the letter $1$.
Write $\tilde u = \st(u_{j+1}u_{j+2}\cdots u_{m+1})$ and $\tilde v = \st(v_1v_2\cdots v_{k})$.
Assume by induction that 
if $u' \in S_{m'+1}$ and $v' \in S_{n'+1}$ where $m' +n' < m + n$
then $\nabla_\shuffle([u']\otimes[v']) = \sum_{\pi \in \cS_\shuffle(u',v')} [\pi]$.
Fix $a \in \cR(u)$ and $b \in \cR(v)$.
Since no two wires in a reduced wiring diagram cross twice,
the $j$th wire of the standard wiring diagram of $a$ 
is monotonically increasing and equal to $m+1$ at $x=N$, and is an upper bound 
for wires $1,2,\dots,j-1$.
Similarly, the $(k+1)$th wire of the standard wiring diagram of $b$
is monotonically decreasing and is equal to $1$ at $x=N$, and is a lower bound for wires $k+1,k+2,\dots,n+1$.

Let $c$ be a word obtained by shuffling $a$ and $b\uparrow m$, i.e., assume $c \cap \{1,2,\dots,m\} = a$ and $c \cap (m+\PP) = b \uparrow m$.
Let $D=(f_1,f_2,\dots,f_m, g_1,g_2,\dots, g_{n+1})$ be the standard (reduced) wiring diagram of $c$
and suppose this is a wiring diagram for the permutation $\pi \in S_{m+n+1}$. 
Let $N =\ell(a) + \ell(b)= \ell(c)$
and note that  $c$ contains both $m$ and $m+1$ as letters. 
If $c \cap \{m,m+1\}$ begins with $m$,
then the wire $f_j$ is monotonically increasing with $f_j(N) = m+1$,
each of the wires $f_1,f_2,\dots,f_{j-1}$ is bounded above by $f_j$,
and we have
$\pi_1\pi_2\cdots \pi_j = u_1u_2\cdots u_j.$
It follows in this case that if $E = (f_{j+1},f_{j+2},\dots,f_{m},g_1,g_2,\dots, g_{n+1})$ 
then 
\[ \word(E) \cap (m-j+\PP) = b \uparrow (m-j) 
\qquand \word(E) \cap \{1,2,\dots,m-j\} \in \cR(\tilde u),\]
so $\pi =u_1u_2\cdots u_j \dfslash \sigma$
for some $\sigma \in \cS_\shuffle(\tilde u, v)$ by induction.
Alternatively, if $c \cap \{m,m+1\}$ begins with $m+1$,
then the wire $g_{k+1}$ is monotonically decreasing with $g_{k+1}(N)=m+1$,
 each of the wires $g_{k+2},g_{k+3},\dots,g_{n+1}$ is bounded below by $g_{k+1}$,
and
we have
$\pi_{m+k+1}\pi_{m+k+2}\cdots \pi_{m+n+1} = w_{k+1}w_{k+2}\cdots w_{n+1}.$
It follows in this case that if $F = (f_{1},f_{2},\dots,f_{m},g_1,g_2,\dots, g_{k})$
then \[ \word(F) \cap \{1,2,\dots,m\} = a  \qquand \word(F) \cap (m+\PP) \in \cR(\tilde v),
\] 
so $\pi = \sigma \dbslash w_{k+1}w_{k+2}\cdots w_{n+1}$
for some $\sigma \in \cS_\shuffle( u, \tilde v)$ by induction.

Since we know that the set of all shuffles of $a \in \cR(u)$ and $b\uparrow m$ for $b \in \cR(v)$
decomposes as a disjoint union of the sets $\cR(\pi)$ for certain permutations $\pi \in S_{m+n+1}$,
the preceding argument shows $\nabla_\shuffle([u]\otimes [v])$ is a multiplicity-free sum of terms $[\pi]$ where $\pi$ ranges over a subset of
$\cS_\shuffle(u,v)$.
To show that every term indexed by $\pi \in \cS_\shuffle(u,v)$ appears in the product, suppose $\sigma \in \cS_\shuffle(\tilde u, v)$ and $\tilde c \in \cR(\sigma)$.
By induction, $\tilde c$ is a shuffle of $\tilde a$ and $b \uparrow (m-j)$ for some $\tilde a \in \cR(\tilde u)$ and $b \in \cR(v)$.
Let $ d$ be any reduced word for $\tau := \st(u_1u_2\cdots u_j)$ 
and let $ e $ be the word formed by concatenating the sequences $(u_i-1)(u_i-2)\cdots \tau(i)$
as $i \in \{1,2,\dots,j\}$ varies in the order that makes the values $u_i$ increasing.
Then $a := e  d (\tilde a \uparrow j)$ is a reduced word for $u= u_1u_2\cdots u_j \dfslash \tilde u$, while $c := e  d (\tilde c\uparrow j)$ is a reduced word for 
$u_1u_2\cdots u_j \dfslash \sigma$ and also a shuffle of $a$ and $b\uparrow m$.
We conclude that $[u_1u_2\cdots u_j \dfslash \sigma]$ appears in the product $\nabla_\shuffle([u]\otimes [v])$.

Similarly, suppose $\sigma \in \cS_\shuffle( u, \tilde v)$ and $\tilde c \in \cR(\sigma)$. Then $\tilde c$ is a shuffle of $ a$ and $\tilde b \uparrow m$ for some $ a \in \cR( u)$ and $b \in \cR(\tilde v)$.
Let $ d$ be any reduced word for $\tau := \st(v_{k+1}v_{k+2}\cdots v_{n+1})$ 
and let $ e $ be the word formed by concatenating the sequences $v_i(v_i+1)\cdots (\tau(i-k) + k)$
as $i \in \{k+1,k+2,\dots, n+1\}$ varies in the order that makes the values $v_i$ decreasing.
Then $b := e(d\uparrow k) \tilde b$ is a reduced word for $v= \tilde v\dbslash v_{k+1}v_{k+2}\cdots v_{n+1}$, while $c :=  (e\uparrow m)  (d\uparrow (m+k)) \tilde c$ is a reduced word for 
$ \sigma \dbslash w_{k+1}w_{k+1}\cdots w_{n+1}$ and also a shuffle of $a$ and $b\uparrow m$.
Hence $[\sigma \dbslash w_{k+1}w_{k+2}\cdots w_{n+1}]$ appears in the product $\nabla_\shuffle([u]\otimes [v])$.
This completes our proof that $\nabla_{\shuffle}([u]\otimes [v]) = \sum_{\pi \in \cS_\shuffle(u,v)} [\pi]$.
\end{proof}

\begin{example} We illustrate our proof strategy by computing $\nabla_\shuffle([231]\otimes [312])$.
Let $u = 231$ and $v=312$ so that $m=n=2$ and $w = 534$. The permutations $u$ and $v$ each have one reduced word: $\cR(u) = \{12\}$ and $\cR(v) = \{21\}$.
These words have the following wiring diagrams:
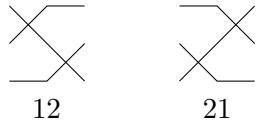
\begin{center}
$\barr{c}
\begin{tikzpicture}[scale=0.5]
\draw (0,0) -- (1,0) -- (2,1);
\draw (0,1) -- (1,2) -- (2,2);
\draw (0,2) -- (1,1) -- (2,0);
\end{tikzpicture}
\\
12
\earr$
\qquad
$\barr{c}
\begin{tikzpicture}[scale=0.5]
\draw (0,0) -- (1,1) -- (2,2);
\draw (0,1) -- (1,0) -- (2,0);
\draw (0,2) -- (1,2) -- (2,1);
\end{tikzpicture}
\\
21
\earr$
\end{center}
We have $u_j = w_{k+1} = m+1$ for $j=k+1=2$.
Let $a = 12 \in \cR(u)$ and $b=21 \in \cR(v)$. There are $6 = \binom{4}{2}$ different shuffles $c$ of $a$ and $b \uparrow m$.
The three wiring diagrams $D = (f_1,f_2,g_1,g_2,g_3)$ of the shuffles $c$ such that $c \cap \{m,m+1\}$ begins with $m=2$ are shown below:
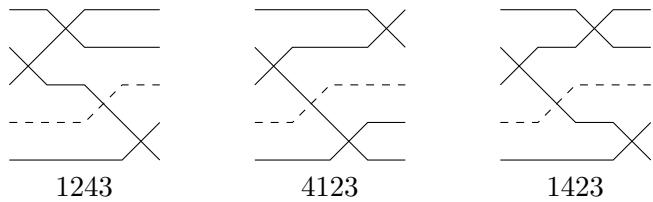
\begin{center}
$\barr{c} 
\begin{tikzpicture}[scale=0.5]
\draw (0,0) -- (1,0) -- (2,0) -- (3,0) -- (4,1);
\draw[dashed] (0,1) -- (1,1) -- (2,1) -- (3,2) -- (4,2);
\draw (0,2) -- (1,3) -- (2,4) -- (3,4) -- (4,4);
\draw (0,3) -- (1,2) -- (2,2) -- (3,1) -- (4,0);
\draw (0,4) -- (1,4) -- (2,3) -- (3,3) -- (4,3);
\end{tikzpicture}
\\
1243
\earr$
\qquad
$\barr{c} 
\begin{tikzpicture}[scale=0.5]
\draw (0,0) -- (1,0) -- (2,0) -- (3,1) -- (4,1);
\draw[dashed] (0,1) -- (1,1) -- (2,2) -- (3,2) -- (4,2);
\draw (0,2) -- (1,3) -- (2,3) -- (3,3) -- (4,4);
\draw (0,3) -- (1,2) -- (2,1) -- (3,0) -- (4,0);
\draw (0,4) -- (1,4) -- (2,4) -- (3,4) -- (4,3);
\end{tikzpicture}
\\
4123
\earr$
\qquad
$\barr{c}
\begin{tikzpicture}[scale=0.5]
\draw (0,0) -- (1,0) -- (2,0) -- (3,0) -- (4,1);
\draw[dashed] (0,1) -- (1,1) -- (2,2) -- (3,2) -- (4,2);
\draw (0,2) -- (1,3) -- (2,3) -- (3,4) -- (4,4);
\draw (0,3) -- (1,2) -- (2,1) -- (3,1) -- (4,0);
\draw (0,4) -- (1,4) -- (2,4) -- (3,3) -- (4,3);
\end{tikzpicture}
\\
1423
\earr$
\end{center}
In these pictures, the wire $f_j$ is shown as a dashed line. As described in the proof, this wire is monotonically increasing and eventually equal to $m+1$,
and is an upper bound for $f_1,f_2,\dots,f_{j-1}$. Moreover, if $D$ is a wiring diagram for $\pi \in S_{m+n-1}$ then  $\pi_1\cdots \pi_j = u_1\cdots u_j= 23$ as claimed.
Removing the wires $f_1,f_2,\dots,f_j$ from each diagram produces a wiring diagram for $v$, which is the unique element in $\cS_\shuffle(\tilde u, v)$ since $\tilde u = \st(u_{j+1}u_{j+2}\cdots u_{m+1})=1$:
\begin{center}
$\barr{c} 
\begin{tikzpicture}[scale=0.5]
\draw (0,2) -- (1,3) -- (2,4) -- (3,4) -- (4,4);
\draw (0,3) -- (1,2) -- (2,2) -- (3,1) -- (4,0);
\draw (0,4) -- (1,4) -- (2,3) -- (3,3) -- (4,3);
\end{tikzpicture}
\earr$
\qquad
$\barr{c} 
\begin{tikzpicture}[scale=0.5]
\draw (0,2) -- (1,3) -- (2,3) -- (3,3) -- (4,4);
\draw (0,3) -- (1,2) -- (2,1) -- (3,0) -- (4,0);
\draw (0,4) -- (1,4) -- (2,4) -- (3,4) -- (4,3);
\end{tikzpicture}
\earr$
\qquad
$\barr{c}
\begin{tikzpicture}[scale=0.5]
\draw (0,2) -- (1,3) -- (2,3) -- (3,4) -- (4,4);
\draw (0,3) -- (1,2) -- (2,1) -- (3,1) -- (4,0);
\draw (0,4) -- (1,4) -- (2,4) -- (3,3) -- (4,3);
\end{tikzpicture}
\earr$
\end{center}
The word $1243$ corresponds to the reduced word $c = ed(\tilde c \uparrow j)$ described in the second to last paragraph of the proof of Theorem~\ref{a-shuff-thm}.

Next, consider the wiring diagrams $D = (f_1,f_2,g_1,g_2,g_3)$ for the three words $c$ obtained by shuffling $a=12$ and $b\uparrow m = 43$ such that $c \cap \{m,m+1\}$ begins with $m+1=3$:
\begin{center}
$\barr{c}
\begin{tikzpicture}[scale=0.5]
\draw (0,0) -- (1,0) -- (2,1) -- (3,1) -- (4,1);
\draw (0,1) -- (1,2) -- (2,2) -- (3,3) -- (4,4);
\draw (0,2) -- (1,1) -- (2,0) -- (3,0) -- (4,0);
\draw[dashed] (0,3) -- (1,3) -- (2,3) -- (3,2) -- (4,2);
\draw (0,4) -- (1,4) -- (2,4) -- (3,4) -- (4,3);
\end{tikzpicture}
\\
4312
\earr$
\qquad
$\barr{c}
\begin{tikzpicture}[scale=0.5]
\draw (0,0) -- (1,0) -- (2,0) -- (3,0) -- (4,1);
\draw (0,1) -- (1,2) -- (2,3) -- (3,4) -- (4,4);
\draw (0,2) -- (1,1) -- (2,1) -- (3,1) -- (4,0);
\draw[dashed] (0,3) -- (1,3) -- (2,2) -- (3,2) -- (4,2);
\draw (0,4) -- (1,4) -- (2,4) -- (3,3) -- (4,3);
\end{tikzpicture}
\\
1432
\earr$
\qquad
$\barr{c}
\begin{tikzpicture}[scale=0.5]
\draw (0,0) -- (1,0) -- (2,0) -- (3,1) -- (4,1);
\draw (0,1) -- (1,2) -- (2,3) -- (3,3) -- (4,4);
\draw (0,2) -- (1,1) -- (2,1) -- (3,0) -- (4,0);
\draw[dashed] (0,3) -- (1,3) -- (2,2) -- (3,2) -- (4,2);
\draw (0,4) -- (1,4) -- (2,4) -- (3,4) -- (4,3);
\end{tikzpicture}
\\
4132
\earr$
\end{center}
In these pictures, the wire $g_{k+1}$ is shown as a dashed line. This wire is monotonically decreasing and eventually equal to $m+1$,
and is a lower bound for $g_{k+2},\dots,g_{n+1}$. In turn, if $D$ is a wiring diagram for $\pi \in S_{m+n-1}$ then  $\pi_{m+k+1}\cdots \pi_{m+n+1}= w_{k+1}\cdots w_{n+1} = 34$ as claimed,
and it is easy to see that removing the wires $g_{k+1},g_{k+2},\dots,g_{n+1}$ from each diagram produces a wiring diagram for $u$,
which is the unique element of $\cS_\shuffle(u,\tilde v)$ since $\tilde v = \st(v_1v_2\cdots v_k) = 1$:
\begin{center}
$\barr{c}
\begin{tikzpicture}[scale=0.5]
\draw (0,0) -- (1,0) -- (2,1) -- (3,1) -- (4,1);
\draw (0,1) -- (1,2) -- (2,2) -- (3,3) -- (4,4);
\draw (0,2) -- (1,1) -- (2,0) -- (3,0) -- (4,0);
\end{tikzpicture}
\earr$
\qquad
$\barr{c}
\begin{tikzpicture}[scale=0.5]
\draw (0,0) -- (1,0) -- (2,0) -- (3,0) -- (4,1);
\draw (0,1) -- (1,2) -- (2,3) -- (3,4) -- (4,4);
\draw (0,2) -- (1,1) -- (2,1) -- (3,1) -- (4,0);
\end{tikzpicture}
\earr$
\qquad
$\barr{c}
\begin{tikzpicture}[scale=0.5]
\draw (0,0) -- (1,0) -- (2,0) -- (3,1) -- (4,1);
\draw (0,1) -- (1,2) -- (2,3) -- (3,3) -- (4,4);
\draw (0,2) -- (1,1) -- (2,1) -- (3,0) -- (4,0);
\end{tikzpicture}
\earr$
\end{center}
The word $4312$ corresponds to the reduced word $c =  (e\uparrow m)  (d\uparrow (m+k)) \tilde c$ described in the last paragraph of the proof of Theorem~\ref{a-shuff-thm}.
We conclude that 
\[\nabla_\shuffle([231]\otimes[312]) = \sum_{\pi \in \cS_\shuffle(\tilde u,v)} [\pi] + \sum_{\pi \in \cS_\shuffle(u,\tilde v)}[\pi] = \sum_{\pi \in \cS_\shuffle(u,v)} [\pi]
= [23514] + [25134]. 
\]
\end{example}

%
%
%

Let $X\subset S_{m+1}$ and $Y \subset S_{n+1}$.
Since we can recover $(u,v) \in S_{m+1}\times S_{n+1}$ from $\pi \in \cS_\shuffle(u,v)$
by intersecting its reduced words with $\{1,2,\dots,m\}$ and $m + \{1,2,\dots,n\}$,
 the sets $\cS_\shuffle(u,v)$ are disjoint for all $u \in X$ and $v \in Y$.
Define $\cS_\shuffle(X,Y) = \bigsqcup_{(u,v) \in X\times Y} \cS_\shuffle(u,v) \subset S_{m+n+1}$
and $\cS_\shuffle(u,Y) = \cS_\shuffle(\{u\},Y)$ and $\cS_\shuffle(X,v) = \cS_\shuffle(X,\{v\})$.
The associativity of $\nabla_\shuffle$ implies the following:

\begin{corollary}
If $(\pi^1,\pi^2,\pi^3) \in S_{l+1}\times S_{m+1}\times S_{n+1}$ then $\cS_\shuffle(\pi^1,\cS_\shuffle(\pi^2,\pi^3)) = \cS_\shuffle(\cS_\shuffle(\pi^1,\pi^2),\pi^3)$.
\end{corollary}

For permutations $\pi^1,\pi^2,\dots,\pi^k$, let $\cS_\shuffle(\pi^1,\pi^2,\dots,\pi^k) = \cS_\shuffle(\pi^1, \cS_\shuffle(\pi^2,\dots,\pi^k))$
where $\cS_\shuffle(\pi) = \{\pi\}$.  
Using this notation, we can answer a natural question about what happens when we apply the flattening map $\st$ 
to the reduced words of a permutation.

A permutation  $\pi = \pi_1\pi_2\cdots \pi _n \in S_n$ is \emph{irreducible} if 
$\{\pi_1,\pi_2,\dots,\pi_m\} \neq\{1,2,\dots,m\}$ for all $1\leq m <n$.
If $u \in S_m$ and $v \in S_n$, then define $u \oplus v = u (v \uparrow m)  \in S_{m+n}$.
Clearly $\pi \oplus(\pi' \oplus\pi'') = (\pi \oplus \pi') \oplus \pi''$,
so we can omit all parentheses in expresssions using $\oplus$. Moreover, every $\pi \in S_n$
has a unique factorization $\pi = \pi^1 \oplus \pi^2 \oplus \cdots \oplus \pi^k$ where each $\pi^i \in S_1 \sqcup S_2 \sqcup S_ 3 \sqcup \cdots$ is irreducible.

\begin{corollary}
Suppose $\pi = \pi^1 \oplus \pi^2 \oplus \cdots \oplus \pi^k \in S_n$ where each $\pi^i$ is an irreducible permutation.
The flattening map $\st$ is then a bijection $\cR(\pi) \to \bigsqcup_{\sigma \in \cS_\shuffle(\pi^1,\pi^2,\dots,\pi^k)} \cR(\sigma)$.
\end{corollary}

\begin{proof}
The map $\st$ is certainly injective,
and if $n_i \in \NN$ is such that $\pi^i \in S_{n_i + 1}$ ,
then $\st(\cR(\pi))$ is the disjoint union of the reduced words of
 the permutations $\sigma \in S_{n_1 + n_2 + \dots +n_k +1}$ for which $[\sigma]$ is a term in the product 
 $\nabla_\shuffle^{(k-1)}([\pi^1] \otimes [\pi^2]\otimes \cdots \otimes [\pi^k])$,
 which by Theorem~\ref{a-shuff-thm} is $\sum_{\sigma \in \cS_\shuffle(\pi^1,\pi^2,\dots,\pi^k)} [\sigma]$.
\end{proof}

\begin{example}
We have $\cR(231645)=\cR(231 \oplus 312) = \{ 1254, 1524, 1542, 5124, 5142, 5412\}$
and  $\{ 1243, 1423, 1432, 4123, 4132, 4312\} = \cR( 23514) \sqcup \cR(25134) = \bigsqcup_{\sigma \in \cS_\shuffle(231,312)} \cR(\sigma)$.
\end{example}

A permutation is \emph{fully commutative} if none of its reduced words contains a consecutive subword of the form $i(i+1)i$.
It is well-known \cite[Theorem 2.1]{BJS} that 
$\pi  \in S_n$ is fully commutative if and only if $\pi$ is \emph{321-avoiding}
in the sense that no indices $i<j<k$ have $\pi_i > \pi_j > \pi_k$.

\begin{corollary}\label{321-cor}
Let $u \in S_{m}$ and $v \in S_{n}$ be fully commutative. Set $w= v \uparrow (m-1)$.
Form $u' \in S_{m+n-1}$ from $uw$ by replacing the second occurrence of $m$ by $u_{m}$ and then removing the $m$th letter.
Form $v' \in S_{m+n-1}$ from $uw$ by replacing the first occurrence of $m$ by $w_{1}$ and then removing the $(m+1)$th letter.
If $u'=v'$ then $\nabla_{\shuffle}([u]\otimes [v]) = [u'] = [v']$; otherwise $\nabla_{\shuffle}([u]\otimes [v])  = [u'] + [v']$.
\end{corollary}

For example, $\nabla_{\shuffle}([4123]\otimes [2341]) = [4125673] + [5123674]$.

\begin{proof}
This follows from Theorem~\ref{a-shuff-thm} since if $u$ and $v$ are 321-avoiding,
then the elements $\tilde u$ and $\tilde v$ in Definition~\ref{a-shuffle-def}
must both be identity permutations.
\end{proof}

Let $\bfPi^{\FC}$ be the space spanned by $[\pi]$ for all fully commutative elements $\pi \in S_{n}$ and $n \in \PP$.
\begin{corollary}
The subspace $\bfPi^{\FC} $ is a graded sub-bialgebra of $(\bfPi,\nabla_\shuffle,\iota_\shuffle,\Delta_\odot,\epsilon_\odot)$.
\end{corollary}

\begin{proof}
Clearly $\Delta_\odot(\bfPi^{\FC}) \subset \bfPi^{\FC} \otimes \bfPi^{\FC}$,
and $\nabla_\shuffle(\bfPi^{\FC} \otimes \bfPi^{\FC}) \subset \bfPi^\FC$ holds by Corollary~\ref{321-cor}.
\end{proof}

\section{Combinatorial bialgebras}\label{qsym-sect}

A \emph{composition} $\alpha$ of $n \in \NN$, written $\alpha \vDash n$,
is a sequence of positive integers $\alpha=(\alpha_1,\alpha_2,\dots,\alpha_l)$
with $\alpha_1 + \alpha_2 + \dots + \alpha_l = n$. 
The unique composition of $n=0$ is the empty word $\emptyset$.
Let $\kk[[x_1,x_2,\dots]]$ be the algebra of formal power series with coefficients in $\kk$ in a countable set of commuting variables.
The \emph{monomial quasi-symmetric function} $M_\alpha$ indexed by a
composition $\alpha\vDash n $ with $l$ parts
is 
\[ M_\alpha = \sum_{i_1<i_2<\dots<i_l} x_{i_1}^{\alpha_1} x_{i_2}^{\alpha_2}\cdots x_{i_l}^{\alpha_l} 
\in \kk[[x_1,x_2,\dots]].\]
When $\alpha$ is the empty composition, set $M_\emptyset=1$.
For each $n \in \NN$, the set $\{ M_\alpha : \alpha \vDash n\}$
is a basis for a subspace  $\QSym_n\subset \kk[[x_1,x_2,\dots]]$.
The vector space of \emph{quasi-symmetric functions}
 $\QSym=\bigoplus_{n \in \NN} \QSym_n$
is a commutative subalgebra of $\kk[[x_1,x_2,\dots]]$.
This algebra is a graded Hopf algebra whose
coproduct 
and counit  are the linear maps with
 $ \Delta(M_\alpha)  = \sum_{\alpha = \beta \gamma} M_\beta \otimes M_\gamma$
and with
$\epsilon(M_\emptyset) = 1$ and $\epsilon(M_\alpha) =0 $ 
 for all compositions $\alpha \neq \emptyset$; see
\cite[\S3]{ABS}.

Each $\alpha \vDash n$
can be rearranged to form a partition  of $n$, denoted $ \sort(\alpha)$.
The \emph{monomial symmetric function} indexed by a partition $\lambda$
is $m_\lambda = \sum_{\sort(\alpha) = \lambda} M_\alpha.$
Write $\lambda \vdash n$ when $\lambda$ is a partition of $n$ and 
let $\Sym_n = \kk\spanning\{ m_\lambda : \lambda \vdash n\}$.
The subspace
$\Sym = \bigoplus_{n \in \NN} \Sym_n \subset \QSym$
is the familiar graded Hopf subalgebra of \emph{symmetric functions}.

\begin{definition}\label{cc-def}
A pair $(A,\zeta)$ is a \emph{combinatorial coalgebra}
if $(A,\Delta,\epsilon)$ is a graded coalgebra 
and $\zeta : A \to \kk$ is a linear map with $\zeta(a) = \epsilon(a)$ whenever $a\in A$ is homogeneous of degree zero.
A morphism of combinatorial coalgebras $(A,\zeta) \to (A',\zeta')$ is a graded coalgebra morphism $\Phi : A \to A'$
with $\zeta = \zeta'\circ \Phi$. 
A combinatorial coalgebra $(A,\zeta)$ in which $A$ is a graded bialgebra and $\zeta$ is an algebra morphism
is a \emph{combinatorial bialgebra}.
Morphisms of combinatorial bialgebras are morphisms of combinatorial coalgebras that are bialgebra morphisms.
\end{definition}

The objects in this definition are mild generalizations of combinatorial coalgebras 
and Hopf algebras as introduced in \cite{ABS}, which restricts to the case when $A$ is connected and has finite graded dimension.
For similar definitions of ``combinatorial'' monoidal structures in other categories, see, for example, \cite[\S5.4]{M0} and \cite{M2}.
Let
$\zetaq : \QSym \to \kk $ 
be the linear map with 
\[ 
\zetaq(M_\alpha) = \begin{cases} 1 &\text{if }\alpha =\emptyset \text{ or }\alpha = (n) \text{ for some }n \in \NN \\
0&\text{if $\alpha$ is any other composition}.\end{cases}
\]
Then $\zetaq$ is the restriction of the algebra morphism $\kk[[x_1,x_2,\dots]]\to \kk$
setting $x_1=1$ and $x_2=x_3=\dots =0$,
so $(\QSym,\zetaq)$ is a combinatorial bialgebra.

Let $\XX(\bfW)$ denote the set of linear maps
$\zeta : \bfW \to \kk$ 
for which $(\bfW,\zeta)$ is a combinatorial bialgebra,
i.e.,
with $\zeta([\emptyset,n]) = 1$ for all $n \in \NN$ and $\zeta \circ \nabla_\shuffle = \nabla_{\kk} \circ (\zeta\otimes \zeta)$.
If $\zeta \in \XX(\bfW)$ then $(\bfPi,\zeta)$ is also a combinatorial bialgebra.

For $l \in \NN$, let $\pi_l : \bfW \to \bfW$ be the natural projection onto the subspace of degree $l$ elements,
that is, the  linear map with 
$\pi_l([w,n]) = [w,n]$ if $\ell(w) = l$ and $\pi_l([w,n]) =0$ otherwise. 
Given a linear map $\zeta : \bfW \to \kk$ 
and a composition $\alpha =(\alpha_1,\alpha_2,\dots,\alpha_m)\neq \emptyset$, 
define $\zeta_\alpha$ to be the map
\be\label{zeta-def}
\bfW \xrightarrow{\ \Delta_\odot^{(m-1)}\ } \bfW^{\otimes m} \xrightarrow{\  \pi_{\alpha_1} \otimes \pi_{\alpha_2}\otimes \cdots \otimes \pi_{\alpha_m}\ } \bfW^{\otimes m} \xrightarrow{\ \zeta^{\otimes m}\ } \kk^{\otimes m} \xrightarrow{\ \nabla_\kk^{(m-1)}\ } \kk.
\ee
For the empty composition, let $\zeta_\emptyset = \epsilon_\odot$.


\begin{theorem}\label{unique-thm}
Suppose $\zeta \in \XX(\bfW)$.
The linear map
$\Psi : \bfW \to \QSym$ 
with
\be\label{psi-def} \Psi(x)= \sum_{\alpha} \zeta_\alpha(x) M_\alpha
\ee
for $x \in \bfW$, 
where the sum is over all compositions,
is then the unique morphism of combinatorial bialgebras $(\bfW,\zeta) \to (\QSym,\zetaq)$
and
$(\bfPi,\zeta) \to (\QSym,\zetaq)$.
\end{theorem}

The sum defining $\Psi$ has only finitely many nonzero terms
since $\zeta_\alpha([w,n]) \neq 0$ implies $\alpha \vDash \ell(w)$.

\begin{proof}
Both $(\bfW_n,\zeta)$ and $(\bfPi_n,\zeta)$ are combinatorial coalgebras in the sense of \cite{ABS},
which is slightly more specific than our definition.
By \cite[Theorem 4.1]{ABS}, $(\QSym,\zetaq)$ is the terminal object in the category of combinatorial coalgebras
and
 \eqref{psi-def}
is the unique morphism $(\bfW_n,\zeta) \to (\QSym,\zetaq)$; 
this is a bialgebra morphism $\bfW \to \QSym$
since the maps
\[\bfW_m \otimes \bfW_n \xrightarrow{\ \Psi\otimes \Psi\ } \QSym \otimes \QSym \xrightarrow{\ \nabla\ } \QSym
\quand
\bfW_m \otimes \bfW_n \xrightarrow{\ \nabla_\shuffle\ } \bfW_{m+n} \xrightarrow{\ \Psi\ } \QSym\]
are both morphisms $(\bfW_m\otimes \bfW_n, \nabla\circ (\zeta\otimes \zeta)) \to (\QSym,\zetaq)$.
The restriction of $\Psi$ is a morphism of combinatorial bialgebras $(\bfPi,\zeta) \to (\QSym,\zetaq)$, so must be the unique such morphism.
\end{proof}

Theorem~\ref{unique-thm} is a special case of the following more general fact:

\begin{theorem}[\cite{ABS}; see \cite{M2}] \label{unproved-thm}
Suppose $(B,\zeta)$ is a combinatorial bialgebra.
Define $\Psi : B \to \QSym$ by the formula \eqref{psi-def}, where 
$\zeta_\alpha : B \to \kk$ is the map \eqref{zeta-def}
with $\bfW$ and $\Delta_\odot$ replaced by $B$ and its coproduct.
Then $\Psi$ is the unique morphism of combinatorial bialgebras $(B,\zeta)\to(\QSym,\zetaq)$.
\end{theorem}

This is a minor extension of \cite[Theorem 4.1]{ABS} and \cite[Remark 4.2]{ABS},
and holds by essentially the same arguments; a detailed proof is included in \cite[\S7]{M2}.

The are four particularly natural elements $\zeta \in \XX(\bfW)$, for which we can describe the map $\Psi$ explicitly.
First, let 
$\zeta_{\leq}: \bfW \to \kk$
be the linear map  with
\be
\label{same-eq}
\zeta_{\leq}([w,n]) =  \begin{cases} 1 &\text{if $w$ is weakly increasing}
\\
0&\text{otherwise}
\end{cases}
\qquad\text{for $[w,n] \in \Words$.}
\ee
Define 
$ \zeta_{\geq}$, $\zeta_{<}$, $\zeta_>$ to be the linear maps $\bfW \to \kk$
given by the same formula but with ``weakly increasing''
replaced by ``weakly decreasing,'' ``strictly increasing,'' and ``strictly decreasing.'' 

\begin{proposition}
For each symbol $\bullet \in \{ {\leq}, {\geq}, {<}, {>}\}$,
it holds that $\zeta_\bullet \in \XX(\bfW)$.
\end{proposition}

\begin{proof}
It suffices to check that $\zeta_\bullet\circ \nabla_\shuffle([v,m]\otimes [w,n]) =
\zeta_\bullet([v,m])\zeta_\bullet([w,n])$, which is routine.
\end{proof}

Given  $\alpha = (\alpha_1,\alpha_2,\dots,\alpha_l) \vDash n$, let
$I(\alpha) = \{\alpha_1, \alpha_1 + \alpha_2, \dots ,\alpha_1 +\alpha_2+ \dots + \alpha_{l-1}\}.$
The map $\alpha \mapsto I(\alpha)$
is a bijection from compositions of $n$ to subsets of $[n-1]$.
Write $\alpha \leq \beta$ if $\alpha,\beta\vDash n$ and $I(\alpha) \subseteq I(\beta)$.
The \emph{fundamental quasi-symmetric function} associated to $\alpha \vDash n$ is 
\[
L_\alpha = \sum_{\alpha\leq \beta} M_\beta 
=
 \sum_{\substack{i_1 \leq i_2\leq \dots\leq i_n \\ i_j < i_{j+1}\text{ if }j \in I(\alpha)}} x_{i_1}x_{i_2}\cdots x_{i_n}\in \QSym_n.
\]
The set $\{L_\alpha : \alpha \vDash n\}$ is
another basis of $\QSym_n$.
Given $\alpha = (\alpha_1,\alpha_2,\dots,\alpha_l) \vDash n$,
let  $\beta \vDash n$ be such that $I(\beta) = \{1,2,\dots,n-1\} \setminus I(\alpha)$
and
define the \emph{reversal}, \emph{complement}, and \emph{transpose} of $\alpha$ to be
\[\alpha^\r = (\alpha_l,\dots,\alpha_2,\alpha_1),
\qquad
\alpha^\c =\beta,
\qquand \alpha^\t = (\alpha^\r)^\c = (\alpha^\c)^\r.
\] 
For a word $w=w_1w_2\cdots w_n$, define 
$w^\r = w_n\cdots w_2w_1$
and
$\Des(w) = \{ i \in [n-1] : w_i > w_{i+1}\}.$
Finally, for each $\bullet \in \{{\leq}, {\geq}, {<}, {>}\}$, 
let $\Psi_\bullet $ 
be the map \eqref{psi-def} defined relative to $\zeta=\zeta_\bullet$.

\begin{proposition}\label{words-opsi-prop}
If  $[w,n] \in \Words$ and $\alpha \vDash \ell(w)$ are such that $\Des(w) = I(\alpha)$ then
\[
\Psi_{\leq}([w,n]) = L_\alpha,
\quad 
\Psi_{>}([w,n]) = L_{\alpha^\c},
\quad
\Psi_{\geq}([w^\r,n]) = L_{\alpha^\r},
\quand
\Psi_{<}([w^\r,n]) = L_{\alpha^\t}.
\]
\end{proposition}

\begin{proof}
The first identity is well-known and follows by inspecting \eqref{psi-def} and \eqref{same-eq}.
In detail, equations  \eqref{psi-def} and \eqref{same-eq} show that $\Psi_{\leq}([w,n]) = \sum_\alpha M_\alpha$
where the sum is over all compositions whose parts record the lengths of the factors in
a decomposition of $w$ into weakly increasing subwords, that is,
the compositions $\alpha \vDash \ell(w)$ with $\Des(w) \subset I(\alpha)$.
One derives the other identities similarly.
\end{proof}

The set $\XX(\bfW)$ is a monoid with unit element $\iota_\shuffle \circ \epsilon_\odot$
and product $\zeta \zeta ' := \nabla_\kk \circ (\zeta \otimes \zeta') \circ \Delta_\odot$.
For any symbols $\bullet,\circ \in \{ {\leq}, {\geq}, {<}, {>}\}$ define
$\zeta_{\bullet|\circ} = \zeta_\bullet\zeta_\circ \in \XX(\bfW)$
and let $\Psi_{\bullet|\circ}$ be the unique morphism $(\bfW,\zeta_{\bullet|\circ} ) \to (\QSym,\zetaq)$
given by \eqref{psi-def} for $\zeta = \zeta_{\bullet|\circ}$.
For example, if $[w,n] \in \Words$ then
\be
\label{up-down-eq} \zeta_{>|\leq}([w,n]) = \begin{cases} 1 & \text{if }w=\emptyset, \\
2&\text{if }w_1>\dots>w_i \leq w_{i+1} \leq \dots \leq w_m \text{ where $1 \leq i \leq m=\ell(w)$} \\
0&\text{otherwise}.
\end{cases}
\ee
Similar formulas hold for the other possibilities of $\zeta_{\bullet|\circ} $.

One calls $\alpha = (\alpha_1,\alpha_2,\dots,\alpha_l) \vDash n$
a \emph{peak composition} if $\alpha_i\geq 2$ for $1 \leq i < l$, i.e., if $1 \notin I(\alpha)$ and  $i \in I(\alpha)$ $\Rightarrow $ $i\pm 1 \notin I(\alpha)$.
The number of peak compositions of $n$ is the $n$th Fibonacci number.
The \emph{peak quasi-symmetric function} \cite[Proposition 2.2]{Stem} of a peak composition $\alpha \vDash n$ is 
\[
K_\alpha = \sum_{\substack{
\beta \vDash n
\\
I(\alpha) \subset I(\beta) \cup (I(\beta)+1)
}} 2^{\ell(\beta)} M_\beta
=
 \sum_{\substack{i_1 \leq i_2\leq \dots\leq i_n \\ j \notin I(\alpha)\text{ if }i_{j-1}=i_j=i_{j+1}}} 2^{|\{i_1,i_2,\dots,i_n\}| }x_{i_1}x_{i_2}\cdots x_{i_n}
 \in \QSym_n.
\]
If $\alpha = (2,2)\vDash 4$, for example, then $I(\alpha) = \{2\}$
and
 the subsets
$I(\beta)\subset \{1,2,3\}$ that satisfy
$I(\alpha) \subset I(\beta) \cup (I(\beta)+1)$
are $\{1\}$, $\{2\}$, $\{1,2\}$, $\{1,3\}$,  $\{2,3\}$, and $\{1,2,3\}$, so 
\[
K_{(2,2)} = 4 M_{(1,3)} + 4 M_{(2,2)}+ 8M_{(1,1,2)} + 8M_{(1,2,1)} +  8M_{(2,1,1)} + 16M_{(1,1,1,1)} .
\]
If $\kk$ does not have characteristic two,
then the functions $K_\alpha$ (with $\alpha$ a peak composition) are a basis for a graded Hopf subalgebra of $\QSym$ \cite[Proposition 6.5]{ABS},
which we denote by
\[ \OQSym = \kk\spanning\{ K_\alpha : \alpha\text{ is a peak composition} \}.\]
Given a peak composition $\alpha = (\alpha_1,\alpha_2,\dots,\alpha_l)$,
define $\alpha^{\rpk} = (\alpha_l+1, \alpha_{l-1},\dots,\alpha_2,\alpha_1-1)$.

\begin{lemma}\label{flat-lem}
The linear map $\QSym \to \QSym$ with $L_{\beta} \mapsto L_{\beta^\r}$ for all compositions $\beta$
has $K_\alpha \mapsto K_{\alpha^\rpk}$ for each peak composition $\alpha$.
\end{lemma}

\begin{proof}
The given map has $M_{\beta} \mapsto M_{\beta^\r}$ for all $\beta$,
so the result holds since $I(\alpha) \subset I(\beta) \cup (I(\beta)+1)$ if and only if 
$I(\alpha^\rpk) \subset I(\beta^\r) \cup(I(\beta^\r)+1)$ for all $\alpha,\beta\vDash n$
where $\alpha$ is a peak composition.
\end{proof}

For a word $w=w_1w_2\cdots w_n$, define the sets
$\Peak(w) = \{ i : 1<i<n,\ w_{i-1} \leq w_i > w_{i+1}\}$
and
$\Valley(w) = \{ i: 1< i < n,\ w_{i-1} \geq w_i < w_{i+1}\}$.
For $\alpha \vDash n$, let $\Lambda(\alpha)\vDash n$ be the peak composition with
\[I(\Lambda(\alpha)) = \{ i\geq 2 : i \in I(\alpha),\ i-1\notin I(\alpha)\}.\]
If $w$ is a word and $\alpha \vDash \ell(w)$ and $\Des(w) = I(\alpha)$, then $\Peak(w) = I(\Lambda(\alpha))$.

\begin{proposition}\label{words-kpsi-prop}
If  $[w,n] \in \Words$ and  $\alpha,\beta \vDash \ell(w)$ have $\Peak(w) = I(\alpha)$ and $\Valley(w) = I(\beta)$ then
\[
\Psi_{>|\leq}([w,n]) = K_\alpha,
\text{\ \ }
\Psi_{<|\geq}([w,n]) = K_\beta,
\text{\ \ }
\Psi_{\geq|<}([w^\r,n]) = K_{\alpha^{\rpk}},
\text{\ \ and\ \ }
\Psi_{\leq|>}([w^\r,n]) = K_{\beta^{\rpk}}.
\]
\end{proposition}

\begin{proof}
Let 
$\Theta : \QSym \to \OQSym $
be the linear map with 
$\Theta(L_\alpha) = K_{\Lambda(\alpha)}$ for all compositions $\alpha$.
The linear map
$\nu_{\QSym} : \QSym \to \kk$  
whose value at $L_\alpha$ is 1 if $\alpha = \emptyset$, 2 if $\alpha$ has the form $(1,1,\dots,1,k)$, and 0 otherwise
is an algebra morphism,
and $\Theta$ is the unique morphism of combinatorial bialgebras $(\QSym,\nu_{\QSym}) \to (\QSym,\zetaq)$ \cite[Example 4.9]{ABS}.

Using Proposition~\ref{words-opsi-prop}, it is easy to check
that $\zetaq \circ \Theta \circ \Psi_{\leq} = \nu_{\QSym} \circ \Psi_{\leq} = \zeta_{>|\leq} $,
so $ \Theta \circ \Psi_{\leq} $ is a  morphism of combinatorial bialgebras  
$(\bfW, \zeta_{>|\leq}) \to (\QSym,\zetaq)$.
Since there is only one such morphism,
we have $\Psi_{>|\leq} = \Theta \circ \Psi_{\leq}$
which implies the first identity.
Similarly, as $\zetaq \circ \Theta \circ \Psi_{\geq} = \zeta_{<|\geq}$,
we have $\Psi_{<|\geq}([w,n]) = \Theta\circ \Psi_{\geq}([w,n]) = K_{\Lambda(\alpha^\r)}$
for the composition $\alpha \vDash \ell(w)$ such that $I(\alpha) = \Des(w^\r)$. Tracing through the definitions,
one finds that $\Lambda(\alpha)$ is precisely the composition $\beta \vDash\ell(w)$ with $\Valley(w) = I(\beta)$.
The other identities follow  by Lemma~\ref{flat-lem}.
\end{proof}

\section{Stanley symmetric functions}\label{stan-sect}

In this section, we assume $\kk$ has characteristic zero, so that $\kk$
contains the rational numbers $\QQ$.

\subsection{Type A}\label{type-a-sect}

The \emph{Stanley symmetric function} of a permutation $\pi \in S_n$  with $l=\ell(\pi)$ is the power series
\be\label{stan-def}
 F_\pi  = \sum_{w \in \cR(\pi)} \sum_{ \substack{i_1\leq i_2 \leq \dots \leq i_{l} \\  i_j < i_{j+1} \text{ if }w_j<w_{j+1}}} x_{i_1}x_{i_2}\cdots x_{i_{l}}  \in \kk[[x_1,x_2,\dots]].
 \ee
This definition differs from Stanley's original construction in \cite{Stan}
by an inversion of indices, a common convention in subsequent literature.
As the name suggests, one has $F_\pi \in \Sym$ \cite{Stan}; more strongly,
each $F_\pi \in \NN\spanning\{ s_\lambda : \lambda  \vdash \ell(\pi)\}$ is \emph{Schur positive} \cite{EdelmanGreene}.
Among other reasons, these functions are of interest since they are the stable limits of the Schubert polynomials $\fk S_\pi$,
and since the coefficient of any square-free monomial in $F_\pi$ gives the size of $\cR(\pi)$.

The definition \eqref{stan-def} is more canonical than it first appears: $F_\pi$ is the image of $[\pi] \in \bf\Pi$
under the unique morphism of combinatorial bialgebras $(\bfPi,\zeta_>) \to (\QSym,\zetaq)$. To be precise, 
let
$\omega : \Sym \to \Sym$ be the 
linear map 
with $\omega(s_\lambda) = s_{\lambda^T}$, where $\lambda^T$ denotes the usual transpose of a partition.
This map is the restriction 
of the linear map $\QSym \to\QSym$ with $L_{\alpha} \mapsto L_{\alpha^\c}$
for all $\alpha$.

\begin{proposition}\label{stan-prop} If $\pi \in S_n$ then
$\Psi_>([\pi]) =\Psi_\geq([\pi]) =F_\pi$ 
and
$\Psi_<([\pi]) =\Psi_\leq([\pi]) =\omega(F_\pi)=F_{\pi^{-1}}$.
\end{proposition}

The last equality is \cite[Corollary 7.22]{Macdonald}.

\begin{proof}
Let $\pi \in S_n$.
Since every weakly increasing (respectively, decreasing) consecutive subsequence of a reduced word for $\pi$
is strictly increasing (respectively, decreasing),
it is clear from \eqref{psi-def} that
$\Psi_>([\pi]) =\Psi_\geq([\pi])$ and $\Psi_<([\pi]) =\Psi_\leq([\pi])$.
The inner sum in \eqref{stan-def}
is $L_{\alpha^\c}$ for the composition $\alpha \vDash \ell(w)$ with $I(\alpha) =\Des(w)$,
so Proposition~\ref{words-opsi-prop} implies that $F_\pi = \Psi_>([\pi])$.
Since $F_\pi \in \Sym$, it follows that $F_{\pi^{-1}} = \Psi_{<}([\pi])$
and $\omega \circ \Psi_>([\pi]) =  \Psi_\leq([\pi])$,
 which implies the other equalities.
 \end{proof}

Applying $\Psi_>$ to Corollary~\ref{coprod-cor} and Theorem~\ref{a-shuff-thm}
gives the following identities:

\begin{corollary}
If $\pi \in S_n$ then
$\Delta(F_\pi) = \sum_{\pi \overset{\bullet}=\pi'\pi''} F_{\pi'} \otimes F_{\pi''}$.
\end{corollary}


\begin{corollary}\label{prod-cor}
If $\pi' \in S_{m}$ and $\pi'' \in S_{n}$ then
$F_{\pi'}F_{\pi''} = \sum_{\pi \in \cS_\shuffle(\pi',\pi'')} F_\pi$.
\end{corollary}

Let $w_0^{(n)} = n\cdots 321 \in S_{n}$.
It is well-known that $F_{w_0^{(n)}} = s_{\delta_n}$ for 
$\delta_n=(n-1,\dots,3,2,1)$ \cite{Stan}.
Moreover, if $n \in \PP$ and $p=\lfloor \frac{n+1}{2}\rfloor$ and $q= \lceil \frac{n+1}{2}\rceil$,
then
$s_{\delta_p}s_{\delta_q}= P_{(n-1,n-3,n-5,\dots)}$ \cite[Corollary 1.14]{HMP4},
where $P_\lambda \in \Sym$ is the \emph{Schur $P$-function}
indexed by a strict partition $\lambda$ (see \cite[\S A.3]{Stem}).

\begin{corollary}\label{p-cor} If $n \in \PP$ and $p\in \{\lfloor \frac{n+1}{2}\rfloor, \lceil \frac{n+1}{2}\rceil\}$ then
$P_{(n-1,n-3,n-5,\dots)} = \sum_{\pi \in \cB(p,n)} F_\pi.$
\end{corollary}

\begin{proof}
This follows by applying $\Psi_>$ to Theorem~\ref{intro-thm}(a) with $m=p\in \{\lfloor \frac{n+1}{2}\rfloor, \lceil \frac{n+1}{2}\rceil\}$.
\end{proof}

\begin{example}
For $n=5$, we have $P_{(4,2)} = 
F_{32541} +F_{34512} + F_{35142} + F_{42513} + F_{45123} + F_{52143}.$
\end{example}

\subsection{Types B, C, and D}\label{last-subsect}

In a similar way,
we can realize the {Stanley symmetric functions of types B, C, and D} from \cite{BH,FominKirillov,TKLamThesis,TKLam}
as the images of certain canonical morphisms from algebraic structures on signed permutations.

Given $n \in \PP$, let $B_n$ denote the group of \emph{signed permutations}
of $\{\pm 1, \pm 2,\dots,\pm n\}$,
that is, permutations $\pi$ such that $\pi(-i) = -\pi(i)$ for all $i$.
Let $s^B_1, s^B_2,s^B_3,\dots,s^B_n \in B_n$ be the elements
\[ s^B_1 =(-1,1)\qquand s^B_i = (-i,-i+1)(i-1,i)\quad \text{for $2\leq i \leq n$}.\]
Let $s^D_1,s^D_2,s^D_3,\dots,s^D_n \in B_n$ be the elements
\[ s^D_1 = s^B_1s^B_2s^B_1 = (-2,1)(-1,2) \qquand s^D_i = s^B_i\quad \text{for $2\leq i \leq n$},\]
The normal subgroup $D_n := \langle s_1^D,s_2^D,\dots,s_n^D\rangle  \subset B_n = \langle s_1^B,s_2^B,\dots,s_n^B\rangle$ consists of the signed permutations $\pi$
for which the number of integers $1 \leq i \leq n$ with $\pi(i)<0$ is even.

It is well-known that 
$B_n$ and $D_n$ are the finite Coxeter group of types $B_n$/$C_n$ and $D_n$
relative to the generating sets $\{s_1^B,s_2^B,\dots,s_n^B\}$
and $\{s_1^D,s_2^D,\dots,s_n^D\}$.
Given $\pi \in B_n$,
write $\cR^B(\pi)$ and $\cR^D(\pi)$ 
for the sets of words $i_1i_2\cdots i_l$ of minimal length such that
$\pi = s^B_{i_1}s^B_{i_2}\cdots s^B_{i_l}$
and $\pi = s^D_{i_1}s^D_{i_2}\cdots s^D_{i_l}$, respectively.
We refer to elements of these sets as \emph{reduced words}.
Evidently $\cR^D(\pi)$ is empty unless $\pi \in  D_n$.
Write $\ell^B(\pi)$ for the common value of $\ell(w)$ for all $w \in \cR^B(\pi)$,
and define $\ell^D(\pi)$ for $\pi \in D_n$ analogously.
Explicit formulas for these length functions appear in \cite[Chapter 8]{CCG}.

For  $\pi \in B_n$,
define 
$ [\pi]_B = \sum_{w \in\cR^B(\pi)} [w,n] \in \bfW_n
$
and
$
[\pi]_D =\sum_{w \in \cR^D(\pi)} [w,n] \in \bfW_n
$.
Let 
\[ 
\bfPi^B_n = \kk\spanning \left\{ [\pi]_B : \pi \in B_n \right\} 
\qquand
\bfPi^D_n = \kk\spanning \left\{ [\pi]_D : \pi \in D_n \right\}
\]
 and define 
 \[\bfPi^B = \bigoplus_{n\geq 1} \bfPi^B_n
\qquand \bfPi^D = \bigoplus_{n \geq 2} \bfPi^D_n.\]
Observe that $[\pi]_D = 0$ if $\pi \in B_n - D_n$.
The nonzero elements $[\pi]_B \in \bfPi^B$ and $[\pi]_D \in \bfPi^D$
are homogeneous of degree $\ell^B(\pi)$ and $\ell^D(\pi)$, respectively.

Suppose $B$ is a (graded) bialgebra and $M$ is a (graded) coalgebra.
One says that $M$ is a (graded) right \emph{$B$-module coalgebra} if $M$ is a (graded) right $B$-module
such that the multiplication map $M \otimes B \to M$ is a (graded) coalgebra morphism. 
Since $\bfPi$, $\bfPi^B$, and $\bfPi^D$ are all subspaces of 
the bialgebra $\bfW$,
we can multiply their elements together using $\nabla_\shuffle$.

\begin{theorem}
The subspaces $\bfPi^B$ and $\bfPi^D$ are 
graded sub-coalgebras of $(\bfW,\Delta_\odot,\epsilon_\odot)$
and 
graded right module coalgebras for $(\bfPi,\nabla_\shuffle,\iota_\shuffle,\Delta_\odot,\epsilon_\odot)$.
\end{theorem}

\begin{proof}
By \cite[Theorem 3.3.1]{CCG}, each set 
 $\cR^B(\pi)$ for $\pi \in B_n$ is  an equivalence class
under
the strongest relation with
$vw\sim_B v'w'$ whenever $v\sim_B v'$ and $w\sim_B w'$,
and
such that 
$ij  \sim_B  ji $ if $|j-i| > 1$,
$  i(i+1)i \sim_B  (i+1)i(i+1) 
$
if $i \geq 2$,
and $1212 \sim_B 2121$.
Likewise, each set
$\cR^D(\pi)$ for $\pi \in D_n$ is an equivalence class under 
the strongest relation with
$vw\sim_D v'w'$ whenever $v\sim_D v'$ and $w\sim_D w'$,
and
such that 
$ij  \sim_D  ji $ if $|j-i| > 1$ and $\{i,j\} \neq \{1,3\}$,
$  i(i+1)i \sim_D  (i+1)i(i+1) 
$
if $i \geq 2$,
and $12 \sim_D 21$ and $131\sim 313$.
An equivalence class under the corresponding relation is equal to $\cR^B(\pi)$ or $\cR^D(\pi)$ for some signed permutation $\pi$
if and only if it contains no words with adjacent repeated letters.
It follows that the coproduct $\Delta_\odot$
satisfies $\Delta_\odot(\bfPi^B) \subset \bfPi^B \otimes \bfPi^B$
and
$\Delta_\odot(\bfPi^D) \subset \bfPi^D \otimes \bfPi^D$.

As $\bfW$ is a bialgebra,
to show that 
$\bfPi^B$ and $\bfPi^D$ are 
graded right $\bfPi$-module coalgebras
it suffices to check that $\nabla_\shuffle(\bfPi^B \otimes \bfPi) \subset \bfPi^B$ and
$\nabla_\shuffle(\bfPi^D \otimes \bfPi) \subset \bfPi^D$.
These inclusions follow by arguments similar to the proof of Theorem~\ref{a-shuff-thm}.
\end{proof}

We write $\pi \overset{\bullet}=_B\pi'\pi''$
to mean that $\pi,\pi',\pi'' \in B_n$ and $\pi = \pi'\pi''$ and $\ell^B(\pi) = \ell^B(\pi') + \ell^B(\pi'')$.
Define the operator $\overset{\bullet}=_D$ similarly.
The following is clear from the preceding proof:

\begin{corollary}\label{coprod-bd-cor}
If $\pi \in B_n$ then 
\[
\Delta_{\odot}([\pi]_B) = \sum_{\pi \overset{\bullet}=_B\pi'\pi''} [\pi']_B\otimes [\pi'']_B
\qquand
\Delta_{\odot}([\pi]_D) = \sum_{\pi \overset{\bullet}=_D\pi'\pi''} [\pi']_D\otimes [\pi'']_D
.\]
\end{corollary}

The one-line representation of $\pi \in B_n$ is the word $\pi_1\pi_2\cdots \pi_n$ where $\pi_i = \pi(i)$, with negative letters indicated as barred entries $\bar{i}$ rather than $-i$.
For example,  
the eight elements of $B_2$ 
are
$ 12,$ $\bar 1 2,$  $1 \bar 2$, $\bar 1 \bar 2$, $2 1$,  $\bar 2 1$, $ 2 \bar 1$, and $\bar 2 \bar1$.
There should be an analogue of Theorem~\ref{a-shuff-thm}
describing the right action of $\bfPi$ on $\bfPi^B$ and $\bfPi^D$, but this remains to be found.
For example, we have
\[
\ba
{[1\bar{2}\bar{3}]_B}   [213] &= [14\bar{3}\bar{2}5]_B + [1\bar{2}4\bar{3}5]_B + [1\bar{2}\bar{4}35]_B + [1\bar{4}\bar{3}25]_B + [4\bar{2}\bar{3}15]_B,
\\
[1\bar{2}\bar{3}]_D  [213] &= [14\bar{3}\bar{2}5]_D + [1\bar{2}4\bar{3}5]_D + [1\bar{2}\bar{4}35]_D + [1\bar{4}\bar{3}25]_D + [4\bar{2}\bar{3}15]_D + [\bar{4}\bar{2}\bar{3}\bar{1}5]_D,
\\[-8pt]
\\
[1\bar{2}\bar{3}]_B  [231] &= [14\bar{3}5\bar{2}]_B + [1\bar{2}45\bar{3}]_B + [1\bar{2}\bar{4}53]_B + [1\bar{4}\bar{3}52]_B + [4\bar{2}\bar{3}51]_B,
\\
[1\bar{2}\bar{3}]_D  [231] &= [14\bar{3}5\bar{2}]_D + [1\bar{2}45\bar{3}]_D + [1\bar{2}\bar{4}53]_D + [1\bar{4}\bar{3}52]_D + [4\bar{2}\bar{3}51]_D + [\bar{4}\bar{2}\bar{3}5\bar{1}]_D,
\\[-8pt]
\\
[1\bar{3}\bar{2}]_B [312] & = [1\bar{3}5\bar{2}4]_B + [1\bar{5}\bar{2}34]_B + [5\bar{3}\bar{2}14]_B,
\\
[1\bar{3}\bar{2}]_D [312] &= [1\bar{3}5\bar{2}4]_D + [1\bar{5}\bar{2}34]_D + [5\bar{3}\bar{2}14]_D + [\bar{5}\bar{3}\bar{2}\bar{1}4]_D,
\\[-8pt]
\\
[1\bar{3}\bar{2}]_B  [321] &= [1\bar{3}54\bar{2}]_B + [1\bar{4}5\bar{2}3]_B + [1\bar{5}\bar{2}43]_B + [4\bar{3}5\bar{2}1]_B + [5\bar{3}\bar{2}41]_B + [5\bar{4}\bar{2}13]_B,
\\
[1\bar{3}\bar{2}]_D  [321] &= [1\bar{3}54\bar{2}]_D + [1\bar{4}5\bar{2}3]_D + [1\bar{5}\bar{2}43]_D + [4\bar{3}5\bar{2}1]_D + [5\bar{3}\bar{2}41]_D + [5\bar{4}\bar{2}13]_D \\& \quad + [\bar{4}\bar{3}5\bar{2}\bar{1}]_D + [\bar{5}\bar{3}\bar{2}4\bar{1}]_D + [\bar{5}\bar{4}\bar{2}\bar{1}3]_D,
\ea
\]
using infix notation $[\pi']_B[\pi'']$ in place of $\nabla_\shuffle([\pi']_B\otimes[\pi''])$.

\begin{problem}\label{problem}
Describe the products $\nabla_\shuffle([\pi']_B\otimes [\pi''])$ and $\nabla_\shuffle([\pi']_D\otimes [\pi''])$.
\end{problem}

There is a module coalgebra version of Definition~\ref{cc-def}:

\begin{definition}
If $(B,\zeta)$ is a combinatorial bialgebra and $(M,\xi)$ is a combinatorial coalgebra
such that $M$ is a graded right $B$-module coalgebra with $\xi(mb) = \xi(m)\zeta(b)$
for all $m \in M$ and $b\in B$,
then we say that $(M,\xi)$ is a \emph{combinatorial $(B,\zeta)$-module coalgebra}.
A morphism of combinatorial $(B,\zeta)$-module coalgebras is a morphism of combinatorial coalgebras that is 
a $B$-module map.
\end{definition}

For any combinatorial bialgebra $(B,\zeta)$,
the pair $(\QSym,\zetaq)$ is a combinatorial $(B,\zeta)$-module coalgebra
for the $B$-module structure given by
$ \QSym \otimes B \xrightarrow{\id \otimes \Psi } \QSym \otimes \QSym \xrightarrow{ \nabla } \QSym$
where $\Psi$ is the unique morphism of combinatorial bialgebras $(B,\zeta) \to (\QSym,\zetaq)$.
This observation relies in general on Theorem~\ref{unproved-thm},
but certainly holds for $B=\bfPi$ without that result.


\begin{theorem}\label{unique-thm2}
Suppose $\zeta^A : \bfPi \to \kk$ and $\zeta^B : \bfPi^B \to \kk$ and $\zeta^D : \bfPi^D \to \kk$
are maps such that 
 $(\bfPi,\zeta^A)$ is a combinatorial bialgebra
and $(\bfPi^B,\zeta^B)$ and $(\bfPi^D,\zeta^D)$ are combinatorial $(\bfPi,\zeta^A)$-module coalgebras.
There are then unique morphisms of combinatorial $(\bfPi,\zeta^A)$-module coalgebras 
$(\bfPi^B,\zeta) \to (\QSym,\zetaq)$ and $(\bfPi^D,\zeta) \to (\QSym,\zetaq)$.
 \end{theorem}

\begin{proof}
Since $(\QSym,\zetaq)$ is the terminal object in the category of combinatorial coalgebras \cite[Theorem 4.1]{ABS},
which contains $(\bfPi^B_n,\zeta^B)$ for all $n \in \PP$,
there exists a unique morphism of combinatorial coalgebras 
$\Psi : (\bfPi^B,\zeta^B) \to (\QSym,\zetaq)$.
This map is a $\bfPi$-module morphism since  
\[ \bfPi^B_m \otimes \bfPi_n \xrightarrow{\ \Psi\otimes \Psi\ }\QSym \otimes \QSym \xrightarrow{\ \nabla\ } \QSym
\quand
\bfPi^B_m \otimes \bfPi_n \xrightarrow{\ \nabla_\shuffle\ }\bfPi^B \xrightarrow{\ \Psi\ } \QSym\]
are both morphisms of combinatorial coalgebras $(\bfPi^B_m \otimes \bfPi_n, \nabla\circ (\zeta^B \otimes \zeta^A)) \to (\QSym,\zetaq)$,
and therefore must coincide. The analogous statement for $(\bfPi^D,\zeta^D)$ follows by the same argument.
\end{proof}

As with Theorem~\ref{unique-thm}, the preceding result may be realized as a special case of a more general principle.
The proof of the following statement requires a slight extension of \cite[Theorem 4.1]{ABS}, applicable to our 
more general definition of combinatorial
coalgebras. This is given in \cite[\S7]{M2}. The result then follows by repeating the preceding proof
with $\bfPi$ and $\bfPi^B$ replaced by $B$ and $M$.

\begin{theorem}
Suppose $(B,\zeta)$ is a combinatorial bialgebra and $(M,\xi)$ is a combinatorial $(B,\zeta)$-module coalgebra.
The unique morphism of combinatorial coalgebras
 $(M,\xi)\to (\QSym,\zetaq)$
 is a $B$-module map, and hence also the unique morphism of  $(B,\zeta)$-module coalgebras.
\end{theorem}

The \emph{type B Stanley symmetric function} of $\pi \in B_n$ with $l=\ell^B(\pi)$ is the power series
\[ F^B_\pi  = \sum_{w \in \cR^B(\pi)} 
\sum_{\substack{
i_1 \leq i_2 \leq \cdots \leq i_l \\ 
j\notin \Peak(w) \text{ if } i_{j-1} = i_j = i_{j+1}
}} 2^{|\{i_1,i_2,\dots,i_l\}| - \ell_0(\pi)} x_{i_1}x_{i_2}\cdots x_{i_l}
\]
where 
$\ell_0(\pi)= |\{ 1\leq i \leq n : \pi(i) < 0\}|$,
which is the number of 1's in every $w \in \cR^B(\pi)$; see \cite[Theorem 2.7]{TKLamThesis}.
The \emph{type C Stanley symmetric function} of $\pi \in B_n$ is 
\[ F^C_\pi = 2^{\ell_0(\pi)} F^B_\pi.\]
(This definition is \cite[Proposition 3.4]{BH}; however, the authors in \cite{BH} refer to $F^C_\pi$ 
as a ``$B_n$ Stanley symmetric function.'')
Finally, the \emph{type D Stanley symmetric function} of $\pi \in D_n$ with $l=\ell^D(\pi)$ is 
\[ F^D_\pi  = \sum_{w \in \cR^D(\pi)} 
\sum_{\substack{
i_1 \leq i_2 \leq \cdots \leq i_l \\ 
j\notin \Peak(w) \text{ if } i_{j-1} = i_j = i_{j+1}
}} 2^{|\{i_1,i_2,\dots,i_l\}| - o_D(w)} x_{i_1}x_{i_2}\cdots x_{i_l}
\]
where $o_D(w)$ denotes the total number of letters in $w$ equal to 1 or 2;
see \cite[Proposition 3.10]{BH}.
The formula for $F^D_\pi$ also makes sense 
when $\pi \in B_n - D_n$, but gives $F^D_\pi = 0$ since $\cR^D(\pi) = \varnothing$. 


The Hopf subalgebra  $\OSym = \Sym \cap \OQSym$ has 
two distinguished bases indexed by strict partitions: the Schur $P$-functions $P_\lambda$ mentioned earlier,
and the \emph{Schur $Q$-functions} given by $Q_\lambda = 2^{\ell(\lambda)} P_\lambda$
 (see \cite[\S A.1]{Stem}).
It holds that $F^B_\pi \in \NN\spanning\{ P_\lambda :\lambda\text{ strict}\}$ \cite{FominKirillov,TKLam},
$F^C_\pi \in \NN\spanning\{ Q_\lambda:\lambda\text{ strict}\}$ \cite[Proposition 3.4]{BH},
and $F^D_\pi \in \NN\spanning\{ P_\lambda:\lambda\text{ strict}\}$ \cite[Proposition 3.12]{BH}.
These symmetric functions are closely related to the Schubert polynomials of classical types B, C, and D \cite{BH}
and to the enumeration of the sets $\cR^B(\pi)$ and $\cR^D(\pi)$ \cite{TKLamThesis}.

Write $o_B(w)$ for the number of letters equal to 1 in a word $w$, so that $o_B(w) = \ell_0(\pi)$ if $w \in \cR^B(\pi)$.
Let $\zeta^C = \zeta_{>|\leq}$ and define $\zeta^B,\zeta^D : \bfW \to \kk$ to be the linear maps with
\[
\zeta^B([w,n]) = 2^{-o_B(w)} \zeta_{>|\leq}([w,n])
\qquand
\zeta^D([w,n]) = 2^{-o_D(w)} \zeta_{>|\leq}([w,n]).
\]
Similarly,
let $\Psi^C = \Psi_{>|\leq}$ and define $\Psi^B, \Psi^D : \bfW \to \QSym$ to be the linear maps with
\[
\Psi^B([w,n]) = 2^{-o_B(w)} \Psi_{>|\leq}([w,n])
\qquand
\Psi^D([w,n]) = 2^{-o_D(w)} \Psi_{>|\leq}([w,n]).
\]
Since $\zeta^C \in \XX(\bfW)$, it is clear that $(\bfPi, \zeta^C)$ is a combinatorial bialgebra.

\begin{proposition}
Let $\sC$ denote the category of combinatorial $(\bfPi, \zeta^C)$-module coalgebras.
\ben
\item[(a)] The pairs $(\bfPi^B, \zeta^B)$, $(\bfPi^B, \zeta^C)$, and $(\bfPi^D, \zeta^D)$
are all objects $\sC$.

\item[(b)]
The map $\Psi^B$ is the unique $\sC$-morphism
$ (\bfPi^B, \zeta^B) \to (\QSym,\zetaq)$.

\item[(c)]
The map $\Psi^C$ is the unique $\sC$-morphism
$ (\bfPi^B, \zeta^C) \to (\QSym,\zetaq)$.

\item[(d)]
The map $\Psi^D$ is the unique $\sC$-morphism
$ (\bfPi^D, \zeta^D) \to (\QSym,\zetaq)$.
\een
\end{proposition}

\begin{proof}
We have $(\bfPi^B, \zeta^C) \in \sC$ since $\zeta^C \in \XX(\bfW)$. Part (c) holds since $\Psi^C$ is 
a morphism of combinatorial bialgebras $(\bfW,\zeta^C) \to(\QSym,\zetaq)$ by definition.
Given these observations and Theorem~\ref{unique-thm2}, the other claims
follow since the linear maps $\kappa^B,\kappa^D: \bfW \to \kk$
with $\kappa^B([w,n])= 2^{-o_B(w)}$ and $\kappa^D([w,n])=2^{-o_D(w)}$
have $\kappa^B \circ \nabla_\shuffle([\pi']_B \otimes [\pi'']) = \kappa^B([\pi']_B)$
if $\pi' \in B_m$ and $\pi'' \in S_n$ and $m\geq 1$,
and
$\kappa^D \circ \nabla_\shuffle([\pi']_D \otimes [\pi'']) = \kappa^D([\pi']_D)$
if $\pi' \in D_m$ and $\pi'' \in S_n$ and $m\geq 2$.
\end{proof}

Finally, we observe that the type B, C, and D Stanley symmetric functions occur as the images of
the unique morphisms of combinatorial module coalgebras described in the previous result.

\begin{proposition}
If $\pi \in B_n$ then $\Psi^B([\pi]_B) = F^B_\pi$ and $\Psi^C([\pi]_B) = F^C_\pi$
and $\Psi^D([\pi]_D) = F^D_\pi$.
\end{proposition}

Of course, if $\pi \in B_n - D_n$ then $\cR^D(\pi) = \varnothing$ so $[\pi]_D = 0$ and $F^D_\pi = 0$.

\begin{proof}
These identities are immediate from the definitions and Proposition~\ref{words-kpsi-prop}.
\end{proof}

This interpretation may be used to recover \cite[Corollaries 3.5 and 3.11]{BH}:

\begin{corollary}[See \cite{BH}]
Let $\pi \in B_n$.
Then 
$F^B_\pi=F^B_{\pi^{-1}}$ and $F^C_\pi = F^C_{\pi^{-1}}$
and $F^D_\pi = F^D_{\pi^{-1}}$.
\end{corollary}

\begin{proof}
No reduced word contains adjacent repeated letters,
so $F^C_{\pi} = \Psi_{>|\leq}([\pi]_B) = \Psi_{\geq|<}([\pi]_B)$.
By Lemma~\ref{flat-lem} and Proposition~\ref{words-kpsi-prop},
 $F^C_{\pi^{-1}}=\Psi_{>|\leq}([\pi^{-1}]_B)$ is the image of $\Psi_{\geq|<}([\pi])$ under the linear involution of $\QSym$
with $L_\beta \mapsto L_{\beta^\r}$.
Since $F^C_{\pi} \in \Sym$, this involution has no effect so $F^B_\pi=F^B_{\pi^{-1}}$ and $F^C_\pi = F^C_{\pi^{-1}}$.
The same argument, \emph{mutatis mutandis}, shows that $F^D_\pi = F^D_{\pi^{-1}}$.
\end{proof}

Applying $\Psi_B$, $\Psi_C$, and $\Psi_D$ to Corollary~\ref{coprod-bd-cor} gives the following:
\begin{corollary}
If $\pi \in B_n$ then the following coproduct formulas hold:
\[\Delta(F^B_\pi) = \sum_{\pi \overset{\bullet}=_B\pi'\pi''} F^B_{\pi'} \otimes F^B_{\pi''},
\quad
\Delta(F^C_\pi) = \sum_{\pi \overset{\bullet}=_B\pi'\pi''} F^C_{\pi'} \otimes F^C_{\pi''},
\quad
\Delta(F^D_\pi) = \sum_{\pi \overset{\bullet}=_D\pi'\pi''} F^D_{\pi'} \otimes F^D_{\pi''}.
\]
\end{corollary}

An analogue of Corollary~\ref{prod-cor} in types B, C, and D 
would follow from this result and a solution to Problem~\ref{problem}.

\end{document}